\documentclass[reqno]{amsart}

\usepackage[english]{babel}
\usepackage{amsthm,  latexsym, enumerate, gastex, amssymb,url,listings,enumitem}

\usepackage[mathcal]{eucal}
\usepackage{graphicx,url}

\renewcommand{\phi}[0]{\varphi}
\renewcommand{\theta}[0]{\vartheta}
\renewcommand{\epsilon}[0]{\varepsilon}

\newcommand{\N}{\text{$\mathbb{N}$}}
\newcommand{\Z}{\text{$\mathbb{Z}$}}

\newcommand{\Mod}[1]{\ (\text{mod} \ #1)}
\DeclareMathOperator{\Ap}{Ap}

\newtheorem{theorem}{Theorem}[section]
\newtheorem{lemma}[theorem]{Lemma}
\newtheorem{conjecture}[theorem]{Conjecture}

\theoremstyle{definition}

\newtheorem{definition}[theorem]{Definition}
\newtheorem{example}[theorem]{Example}

\theoremstyle{remark}
\newtheorem{remark}[theorem]{Remark}

\numberwithin{equation}{section}

\begin{document}

\bibliographystyle{amsplain}

\author{S. Ugolini}
\address{Dipartimento di Matematica, Universit\`{a} degli studi di Trento, Via Sommarive 14, I-38123 (Italy)}
\email{s.ugolini@unitn.it}

\date{}

\title[]
{On permutation numerical semigroups}

\begin{abstract}
In this paper we introduce the notion of $n$-permutation numerical semigroup. While there are just three $2$-permutation numerical semigroups, there are infinitely many $n$-permutation numerical semigroups if $n > 2$. We construct $16$ families of $3$-permutation numerical semigroups and one family of $n$-permutation numerical semigroups. Finally we present some experimental data, which seem to support a conjecture about the classification of $3$-permutation numerical semigroups.
\end{abstract}

\maketitle

\section{Introduction}
A numerical semigroup $G$ is a co-finite submonoid of the monoid of non-negative integers $(\N, +, 0)$ (see \cite{ros_ns} for a comprehensive monograph). 

A set $S \subseteq \N$ generates $G$, namely $G = \langle S \rangle$, if and only if the greatest common divisor $\gcd(S)$ of the elements contained in $S$ is equal to $1$.

We can order the elements of $G$ in such a way that
\begin{equation*}
G = \{g_i : i \in \N \},
\end{equation*}
where any $g_i$ is a non-negative integer, $g_0 = 0$ and $g_i < g_{i+1}$ for any $i \in \N$.

We can  associate with $G$ the strictly increasing sequence
\begin{equation*}
g := (g_0, g_1, \dots) = (g_i)_{i \in \N}.
\end{equation*} 

The sequence $g$ can be reduced modulo a positive integer $l$ defining
\begin{equation*}
g \bmod l := (g_0 \bmod l, g_1 \bmod l, \dots) = (g_i \bmod l)_{i \in \N}.
\end{equation*}
If the value of $l$ is not ambiguous, then we  simply write
\begin{equation*}
\overline{g} = (\overline{g_i})_{i \in \N}.
\end{equation*}

We give the following definition.
\begin{definition}
Let $G := \{g_i : i \in \N \}$ be a numerical semigroup, where any $g_i$ is a non-negative integer, $g_0 = 0$ and $g_i < g_{i+1}$ for any $i \in \N$.

Let $n \in \N^* := \N \backslash \{ 0 \}$ and $S := \{g_1, g_2, \dots, g_n \}$.

We say that $G$ is a \emph{$n$-permutation numerical semigroup} (briefly $n$-permutation semigroup)  if $G = \langle S \rangle$ and for any non-negative integer $k$ the $n$-tuple
\begin{equation*}
(\overline{g_{kn+1}}, \overline{g_{kn+2}}, \dots, \overline{g_{kn+n}}) 
\end{equation*}
contains exactly one representative for each residue class of $\Z / n \Z$. 
\end{definition}
\begin{remark}
Rephrasing the definition, we can associate with any $n$-permutation semigroup $G$ an infinite string $g := (g_i)_{i \in \N^*}$ such that the modular string $
\overline{g} := (\overline{g_i})_{i \in \N^*}$ is obtained through the concatenation of infinitely many strings of $(\Z / n \Z)^n$, called $n$-permutations, each of which containing no repetitions. 

We notice in passing that $\overline{g}$ is ultimately periodic since $G$ is a co-finite submonoid of $(\N, +, 0)$. 
\end{remark}

\begin{example}
Let $G := \langle \{9, 14,15, 16 \} \rangle$. 

The increasing sequence of non-zero elements of $G$ is
\begin{equation*}
g = (9, 14, 15, 16, 18, 23, 24, 25, 27, 28, 29, 30, 31, 32, 33, 34, 36, \rightarrow),
\end{equation*}
which corresponds to the (modulo $4$) sequence 
\begin{equation*}
\overline{g} = (\bar{1}, \bar{2}, \bar{3}, \bar{0}, \bar{2}, \bar{3}, \bar{0}, \bar{1}, \bar{3}, \bar{0}, \bar{1}, \bar{2}, \bar{3}, \bar{0}, \bar{1}, \bar{2}) \circ (\bar{0}, \bar{1}, \bar{2}, \bar{3}) \circ (\bar{0}, \bar{1}, \bar{2}, \bar{3}) \circ \dots,
\end{equation*}
namely
\begin{equation*}
\overline{g} = (\bar{1}, \bar{2}, \bar{3}, \bar{0}) \circ (\bar{2}, \bar{3}, \bar{0}, \bar{1}) \circ (\bar{3}, \bar{0}, \bar{1}, \bar{2})^2 \circ (\bar{0}, \bar{1}, \bar{2}, \bar{3})^{\infty}.
\end{equation*}
\end{example}

The paper is organized as follows.
\begin{itemize}
\item In Section \ref{sec_2_pns} we show that there are just three $2$-permutation semigroups.
\item In Section \ref{sec_3_pns} we construct $16$ families $\{ H_{i,k} \}_{i=1}^{16}$ of $3$-permutation semigroups, namely we show that for any positive integer $k$ the numerical semigroup, whose set of generators $S$ is one of the following, is a $3$-permutation semigroup.
\begin{displaymath}
\begin{array}{|c|c|}

\hline
\text{Family} & {S}\\
\hline 
H_{1,k} & \{ 3k, 3k+1, 6k-1 \}\\ 
\hline 
H_{2,k} & \{ 6k+1, 6k+2, 9k+3 \}\\ 
\hline 
H_{3,k} & \{ 6k+1, 9k+2, 9k+3 \}\\ 
\hline 
H_{4,k} & \{ 6k+1, 6k+3, 6k+5 \}\\ 
\hline 
H_{5,k} & \{ 6k+1, 12k-4, 12k \}\\ 
\hline 
H_{6,k} & \{ 3k+1, 6k-1, 6k \}\\ 
\hline 
H_{7,k} & \{ 3k+2, 3k+3, 3k+4 \}\\ 
\hline 
H_{8,k} & \{ 12k+2, 12k+4, 18k+3 \}\\ 
\hline 
H_{9,k} & \{ 3k+2, 6k+1, 6k+3 \}\\ 
\hline 
H_{10,k} & \{ 6k+3, 6k+5, 12k+4 \}\\ 
\hline 
H_{11,k} & \{ 6k+4, 6k+5, 9k+6  \}\\ 
\hline 
H_{12,k} & \{ 12k+4, 18k+3, 18k+5 \}\\ 
\hline 
H_{13,k} & \{ 6k+5, 9k+6, 9k+7 \}\\ 
\hline 
H_{14,k} & \{ 6k+5, 12k+4, 12k+6 \}\\ 
\hline 
H_{15,k} & \{ 12k+8, 12k+10, 18k+15 \}\\ 
\hline 
H_{16,k} & \{ 12k+8, 18k+13, 18k+15 \}\\ 
\hline 
\end{array}
\end{displaymath} 
\item In Section \ref{sec_n_pns} we present some experimental data about $3$-permutation semigroups. Driven by such experiments we conjecture that any $3$-permutation  semigroup $G$ having multiplicity 
\begin{equation*}
m := \min \{x \in G: x > 0 \}
\end{equation*}
at least equal to $12$ belongs to one of the $16$ families studied in Section \ref{sec_3_pns}. In the last part of the section we construct one family of $n$-permutation  semigroups for any positive integer $n \geq 3$. 
\end{itemize}

\section{Preliminaries}
We introduce some notations we will use in the rest of the paper.
\begin{itemize}
\item If $\{a, b \} \subseteq \N$ with $a \leq b$, then 
\begin{displaymath}
\begin{array}{lcl}
[a, b] & := & \{x \in \N : a \leq x \leq b \},\\
\left[a, b \right]_2 & := & \{x \in [a,b] : x \equiv a \Mod{2} \}.
\end{array}
\end{displaymath}
In particular $[a,a] = [a,a]_2 = \{ a \}$.

\item If $A$ and $B$ are two non-empty subsets of $\N$ such that $a < b$ (resp. $a \leq b$) for any pair $(a,b) \in A \times B$, then we write $A <B$ (resp. $A \leq B$).

\item If $G$ is a numerical semigroup, then 
\begin{equation*}
F(G) := \max \{x: x \in \Z \backslash G \}
\end{equation*}
is called the Frobenius number of $G$.

\item If $n \in G \backslash \{ 0 \}$, then the set
\begin{equation*}
\Ap (G,n) := \{s \in G: s - n \not \in G \}
\end{equation*}
is called the Ap\'ery set of $G$ with respect to $n$.
\end{itemize}

The following relation between the Ap\'ery set and the Frobenius number will be used throughout the paper.

\begin{lemma}
If $G$ is a numerical semigroup and $n \in G \backslash \{ 0 \}$, then
\begin{equation*}
F(G) = \max (\Ap(G,n)) - n.
\end{equation*}
\end{lemma}

The following lemmas will be used repeatedly in the paper (see \cite[Lemma 1]{rod} and \cite[Section 5]{rod}).

\begin{lemma}\label{lem_arith_seq}
Let $k$ and $e$ be two positive integers and 
\begin{equation*}
S:=\{a_i \}_{i =0}^k,
\end{equation*}
where $a_0$ is a positive integer and $a_i := a_0 + i e$ for any $i \in [1,k]$.

Then $x \in G := \langle S \rangle$ if and only if 
\begin{equation*}
x = a_0 q + e r
\end{equation*}
with $\{q, r \} \subseteq \N$ and $0 \leq r \leq k q$. 

Moreover
\begin{equation*}
F(G) = a_0 \left\lfloor \frac{a_0-2}{k} \right\rfloor + e (a_0-1).
\end{equation*}
\end{lemma}

The following fact, whose proof is immediate, will be used repeatedly in Lemmas \ref{family_1} - \ref{family_16}.

\begin{lemma}
Let $(x_0, x_1, x_2)$ be a triple in $\N^3$ such that one of the following holds: 
\begin{itemize}
\item $x_i = x_0 + i$ for $i \in \{1, 2 \}$;
\item $x_i = x_0 + 2 i$ for $i \in \{1, 2 \}$.
\end{itemize}
Then $(\overline{x_0}, \overline{x_1}, \overline{x_2})$ is a $3$-permutation.
\end{lemma}

\section{Classification of $2$-permutation semigroups}\label{sec_2_pns}
The numerical semigroups 
\begin{align*}
G_1 & := \langle 1, 2 \rangle = \N, \\
G_2 & := \langle 2, 3 \rangle = \{2, \rightarrow \},\\
G_3 & := \langle 3, 4 \rangle = \{3, 4, 6, \rightarrow \}
\end{align*}
are the only $2$-permutation semigroups.

Indeed, one can easily check that $G_1$, $G_2$ and $G_3$ are $2$-permutation semigroups.

Now we suppose that $G$ is a $2$-permutation semigroup generated by $S:=\{a,b\}$, where $a$ and $b$ are two coprime positive integers with $2 \leq a < b$. 

According to the definition of $2$-permutation semigroups the following hold: 
\begin{itemize}
\item $a$ and $b$ have different parity, namely $b = a + h$ for some  odd integer $h$;
\item $b < 2a$.
\end{itemize}  

We have that
\begin{align*}
S & = \{a, a+h \},\\
2 S & = \{2a, 2a+h, 2a+ 2h \},\\
3 S & = \{3a, 3a+h, 3a+ 2h, 3a + 3h \},\\
4 S & = \{4a, \dots \}.
\end{align*}

First we suppose that $a$ is odd and $b$ is even. 

We deal with different cases.

\begin{itemize}
\item \emph{Case 1:} $h < \frac{a}{3}$. 

The increasing sequence of the first $10$ elements of $G$ is
\begin{equation*}
(a, a+h, 2a, 2a+h, 2a + 2h, 3a, 3a+h, 3a+2h, 3a+3h,4a),
\end{equation*}
which reads (modulo $2$) as follows:
\begin{equation*}
(\overline{1}, \overline{0}, \overline{0}, \overline{1}, \overline{0}, \overline{1}, \overline{0}, \overline{1}, \overline{0}, \overline{0}).
\end{equation*}
Hence $G$ cannot be a $2$-permutation semigroup.

\item \emph{Case 2:} $\frac{1}{3} a < h < \frac{1}{2} a$. 

The increasing sequence of the first $10$ elements of $G$ is
\begin{equation*}
(a, a+h, 2a, 2a+h, 2a + 2h, 3a, 3a+h, 3a+2h, 4a, 3a+3h),
\end{equation*}
which reads (modulo $2$) as follows:
\begin{equation*}
(\overline{1}, \overline{0}, \overline{0}, \overline{1}, \overline{0}, \overline{1}, \overline{0}, \overline{1}, \overline{0}, \overline{0}).
\end{equation*}

Also in this case $G$ cannot be a $2$-permutation semigroup.

\item \emph{Case 3:} $\frac{1}{2} a < h < a$.

The increasing sequence of the first $8$ elements of $G$ is
\begin{equation*}
(a, a+h, 2a, 2a+h, 3a, 2a + 2h, 3a+h, 4a),
\end{equation*}
which reads (modulo $2$) as follows:
\begin{equation*}
(\overline{1}, \overline{0}, \overline{0}, \overline{1}, \overline{1}, \overline{0},  \overline{0}, \overline{0}).
\end{equation*}

\item \emph{Case 4:} $h = \frac{1}{3} a$.

We notice that  $3 \mid a$ because $h \in \N$.

If $a = 3$, then $G = G_3$.

If $a \not = 3$, then $h > 1$.  Therefore $\gcd (a,b)> 1$ and $G$ is not a numerical semigroup.

\end{itemize}

Now we suppose that $a$ is even and $b$ is odd.

We distinguish two different cases.
\begin{itemize}
\item \emph{Case 1:} $h < \frac{a}{2}$.

The increasing sequence of the first $6$ elements of $G$ is
\begin{equation*}
(a, a+h, 2a, 2a+h, 2a + 2h, 3a),
\end{equation*}
which reads (modulo $2$) as follows:
\begin{equation*}
(\overline{0}, \overline{1}, \overline{0}, \overline{1}, \overline{0}, \overline{0}).
\end{equation*}
Hence $G$ cannot be a $2$-permutation semigroup. 

\item \emph{Case 2:} $\frac{a}{2} < h < a$.

The increasing sequence of the first $6$ elements of $G$ is
\begin{equation*}
(a, a+h, 2a, 2a+h, 3a, 2a + 2h),
\end{equation*}
which reads (modulo $2$) as follows:
\begin{equation*}
(\overline{0}, \overline{1}, \overline{0}, \overline{1}, \overline{0}, \overline{0}).
\end{equation*}
Hence $G$ cannot be a $2$-permutation semigroup.

\item \emph{Case 3:} $h = \frac{a}{2}$.

If $a = 2$, then $G = G_2$. 

If $a > 2$, then $h > 1$ and $\gcd (a,b) > 1$, namely $G$ is not a numerical semigroup. 
\end{itemize}

\section{Sixteen families of $3$-permutation semigroups}\label{sec_3_pns}

In the statements of Lemmas \ref{family_1} - \ref{family_16} we always suppose that $G:=\langle S \rangle$, where 
\begin{equation*}
S :=  \{ a_1, a_2, a_3 \}
\end{equation*}
is a subset of $\N^*$ such that
\begin{equation*}
a_1 < a_2 < a_3.
\end{equation*}

We denote by $g := (g_i)_{i=1}^{\infty}$ the increasing sequence of the positive elements in $G$ and by $\overline{g}$ its  reduction modulo $3$.

Moreover, if $A$ is a subset of $\N$, we denote by $g \cap A$ the subsequence of $g$ formed by the elements of $g$ belonging to $A$ and by $\overline{g \cap A}$ its reduction modulo $3$. 

For example, if 
\begin{align*}
g & := (5,7,9,10,12,14, \rightarrow),\\
A & := \{7, 9, 10, 12 \},
\end{align*}
then
\begin{align*}
g \cap A & = (7,9,10,12),\\
\overline{g \cap A} & = (\bar{1}, \bar{0}, \bar{1}, \bar{0}). 
\end{align*}

\begin{lemma}\label{family_1}
Let $S := \{3 k, 3 k +1, 6 k-1 \}$ for some positive integer $k$ and
\begin{equation*}
H_{1,k} := \cup_{i \in \N} (A_{i,k} \cup B_{i,k}), 
\end{equation*}
where
\begin{align*}
A_{i,k} & := [(2i) 3k  - i, (2i) 3 k + 2i],\\
B_{i,k} & := [(2i+1) 3k - i, (2i+1) 3k  + 2i+1 ],
\end{align*}
for any $i \in \N$.

The following hold.
\begin{enumerate}
\item $H_{1,k}$ is a submonoid of $(\N, +,0)$ containing $S$.
\item $A_{i,k} < B_{i,k}$ for any $i \in [0, k-1]$.
\item $B_{i,k} < A_{i+1,k}$ for any $i \in [0,k-1]$.
\item $[(2(k-1)+1) 3k - (k-1), \infty[ \subseteq H_{1,k}$.
\item $G = H_{1,k}$.
\item $H_{1,k}$ is a $3$-permutation semigroup.
\end{enumerate}
\end{lemma}
\begin{proof}
\begin{enumerate}[leftmargin=*]
\item Since $\{3k, 3k+1 \} = B_{0,k}$ and $6k-1 \in A_{1,k}$ we have that $S \subseteq H_{1,k}$. 

If $\{x, y \} \subseteq H_{1,k}$, then 
\begin{align*}
j_1 (3k) - \left\lfloor \frac{j_1}{2} \right\rfloor \leq x \leq j_1 (3k) + j_1,\\
j_2 (3k) - \left\lfloor \frac{j_2}{2} \right\rfloor \leq y \leq j_2 (3k) + j_2,
\end{align*} 
for some $\{j_1, j_2 \} \subseteq \N$.

We notice that 
\begin{equation*}
\left\lfloor \frac{j_1}{2} \right\rfloor + \left\lfloor \frac{j_2}{2} \right\rfloor \leq \left\lfloor \frac{j_1+j_2}{2}   \right\rfloor.
\end{equation*}

Therefore
\begin{equation*}
(j_1 + j_2) (3k) - \left\lfloor \frac{j_1+j_2}{2} \right\rfloor \leq x + y \leq (j_1+j_2) (3k) + (j_1+j_2),
\end{equation*}
namely
\begin{equation*}
x + y \in A_{\lfloor \frac{j_1+j_2}{2} \rfloor, k} \cup B_{\lfloor \frac{j_1+j_2}{2} \rfloor, k}.
\end{equation*}

\item The assertion follows since 
\begin{equation*}
(2i+1) 3k - i - [(2i) 3k + 2i] = 3k - 3i \geq 3
\end{equation*}
for any $i \in [0,k-1]$.
\item The assertion follows since 
\begin{equation*}
[2(i+1) 3k - (i+1)] - [(2i+1) 3k + 2i+1] = 3k - 3i - 2 \geq 1
\end{equation*}
for any $i \in [0,k-1]$.

\item Let $b:=(2(k-1)+1) 3k - (k-1) \in B_{k-1,k}$. 

If $x \in [b, \infty[$, then there exists a set $\{q, r\} \subseteq \N$ such that
\begin{displaymath}
\begin{cases}
x - b = 3 k q + r \\
0 \leq r < 3k,
\end{cases}
\end{displaymath}
namely
\begin{displaymath}
\begin{cases}
x = b + 3 k q + r \\
0 \leq r < 3k.
\end{cases}
\end{displaymath}
Since $3k \in H_{1,k}$ and  
\begin{equation*}
b + r \in B_{k-1,k} \cup \{2k (3k) - k \} \subseteq B_{k-1, k} \cup A_{k,k},
\end{equation*}
we conclude that $x \in H_{1,k}$.

\item Let $x \in H_{1,k}$. We show that $x \in G$ dealing with two cases.
\begin{itemize}
\item If $x \in [j (3k), j (3k) + j]$, where $j \in \{ 2i, 2i+i \}$ for some $i \in \N$, then $x \in \langle \{3k, 3k+1 \} \rangle$ in accordance with Lemma  \ref{lem_arith_seq}.
\item If $x \in \left[j (3k)  - \lfloor \frac{j}{2} \rfloor, j (3k) - 1 \right]$ for some $j \in \N^*$, then 
\begin{equation*}
x = j (3k) - t
\end{equation*}
for some $t \in \left[ 1, \lfloor \frac{j}{2} \rfloor \right]$. Therefore
\begin{equation*}
x =  (j-2t) 3k + t (6k-1).
\end{equation*}
\end{itemize}

From (1) and (4) we deduce that $H_{1,k}$ is a co-finite submonoid of $(\N, +,0)$, namely $H_{1,k}$ is a numerical semigroup.  

Since
\begin{equation*}
S \subseteq H_{1,k} \subseteq G, 
\end{equation*}
we conclude that $G = H_{1,k}$.

\item For any $i \in \N$ we have that
\begin{align*}
|A_{i,k}| & = 3 i + 1,\\
|B_{i,k}| & = 3 i +2.
\end{align*}

Now let $i \in [0,k-1]$. 
From (3) we deduce that 
\begin{equation*}
|B_{i,k} \cup A_{i+1,k}| = 6 i + 6,
\end{equation*}
namely $3$ divides $|B_{i,k} \cup A_{i+1,k}|$. 

The sequence formed by the two greatest elements in $B_{i,k}$ and by the smallest element of $A_{i+1,k}$ reads as follows (modulo $3$): 
\begin{align*}
(\overline{2i}, \overline{2i+1}, \overline{2i+2}).
\end{align*}
Therefore the elements of $\overline{g \cap (B_{i,k} \cup A_{i+1,k})}$ are obtained via a concatenation of $3$-permutations.

Finally we notice that 
\begin{equation*}
A_{k,k} \subseteq [(2(k-1)+1) 3k - (k-1), \infty[ \subseteq H_{1,k}.
\end{equation*}
Hence we conclude that $H_{1,k}$ is a $3$-permutation semigroup.
\end{enumerate}

\end{proof}

\begin{lemma}\label{family_2}
Let $S := \{6 k + 1, 6 k + 2, 9 k + 3 \}$ for some positive integer $k$ and 
\begin{equation*}
H_{2,k} := \{ 0 \} \cup ( \cup_{i \in \N^*} (A_{i,k} \cup B_{i,k})), 
\end{equation*}
where
\begin{align*}
A_{i,k} & := [(6k+1)i, (6k+1)i+i],\\
B_{i,k} & := [(6k+1)i + (3k+2), (6k+1)i + (3k+2) + i-1],
\end{align*}
for any $i \in \N^*$.

The following hold.
\begin{enumerate}
\item $H_{2,k}$ is a submonoid of $(\N, +,0)$ containing $S$.
\item $A_{i,k} < B_{i,k}$ for any $i \in [1, 3k+1]$.
\item $B_{i,k} \leq A_{i+1,k}$ for any $i \in [1,3k]$.
\item $[(6k+1) (3k) + (3k+2), \infty[ \subseteq H_{2,k}$.
\item $G = H_{2,k}$.
\item $H_{2,k}$ is a $3$-permutation semigroup.
\end{enumerate}
\end{lemma}

\begin{proof}
\begin{enumerate}[leftmargin=*]
\item Since $\{6k + 1, 6k + 2 \} = A_{1,k}$ and $9k + 3 \in B_{1,k}$, we have that $S \subseteq H_{2,k}$.

If $\{x, y \} \subseteq H_{2,k}$, then 
\begin{align*}
(6k+1) i_1 + \epsilon_1 (3k+2)  \leq x \leq (6k+1) i_1 + \epsilon_1 (3k+1) + i_1,\\
(6k+1) i_2 + \epsilon_2 (3k+2) \leq y \leq (6k+1) i_2 + \epsilon_2 (3k+1) + i_2,
\end{align*}
for some $\{i_1, i_2 \} \subseteq \N^*$ and $\{\epsilon_1, \epsilon_2 \} \subseteq \{0, 1 \}$.

If $\epsilon_1 = \epsilon_2 = 0$, then $x+y \in A_{i_1+i_2,k}$.

If $\epsilon_1 = \epsilon_2 = 1$, then $x+y \in A_{i_1+i_2+1,k}$ because
\begin{align*}
x + y & \geq (6k+1) (i_1+i_2)  + 2 (3k+2) \\
& = (6 k + 1) (i_1 + i_2 + 1) + 3
\end{align*}
and
\begin{align*}
x + y & \leq (6k+1) (i_1+i_2) + 2 (3k+1) + (i_1+i_2)\\
& = (6k+1) (i_1+i_2+1) + (i_1+i_2+1). 
\end{align*}

If $\epsilon_1 \not = \epsilon_2$, then $x + y \in B_{i_2+i_2,k}$.

\item The assertion follows since 
\begin{equation*}
(6k+1)i + (3k+2) - [(6k+1)i+i] = 3k+2-i \geq 1
\end{equation*}
for any $i \in [1,3k+1]$.

\item The assertion follows since 
\begin{equation*}
(6k+1) (i+1) - [(6k+1)i+ (3k+2) + i - 1] = 3k - i \geq 0
\end{equation*}
for any $i \in [1,3k]$.

\item The proof is as in Lemma \ref{family_1} (4).

\item Let $x \in H_{2,k} \backslash \{ 0 \}$. We show that $x \in G$ dealing with two cases.
\begin{itemize}
\item If $x \in A_{i,k}$ for some $i \in \N^*$, then $x \in \langle \{6k+1, 6k+2 \} \rangle$ in accordance with Lemma \ref{lem_arith_seq}. Hence $x \in G$.
\item If $x \in B_{i,k}$ for some $i \in \N^*$, then 
\begin{align*}
x & = (6k+1) i + (3k+2) + j\\
& = (6k+1) (i-1) + j + (9k+3)
\end{align*}
for some $j \in [0,i-1]$. 

Since 
\begin{equation*}
(6k+1) (i-1) + j \in \langle \{6k+1, 6k+2 \} \rangle, 
\end{equation*}
we conclude that $x \in G$. 
\end{itemize}
Therefore $G = H_{2,k}$ as explained in Lemma \ref{family_1}.

\item We notice that
\begin{equation*}
H_{2,k} = \left( \cup_{i=1}^{3k} (A_{i,k} \cup B_{i,k}) \right) \cup [(6k+1) (3k) + (3k+2), \infty[
\end{equation*}
and $B_{3k,k} \subseteq [(6k+1) (3k) + (3k+2), \infty[$.

For any $i \in \N^*$ we have that
\begin{align*}
|A_{i,k}| & = i+1,\\
|B_{i,k}| & = i.
\end{align*}
Moreover, if $i_1, i_2$ and $i_3$ are three consecutive positive integers such that 
\begin{equation*}
i_1 < i_2 < i_3 \leq 3k
\end{equation*}
where $i_j \equiv j \Mod{3}$ for any $j \in \{1, 2, 3 \}$, then
\begin{equation*}
\sum_{j=1}^3 |A_{i_j, k} \cup B_{i_j,k}| \equiv 2(i_1+i_2+i_3) \equiv 0 \Mod{3},
\end{equation*}
namely $|\cup_{j=1}^3 (A_{i_j, k} \cup B_{i_j,k})|$ is divisible by $3$. 

The sequence formed by the two greatest elements of $A_{i_1}$ and the smallest element of $B_{i_1}$ reads as follows (modulo $3$):
\begin{equation*}
(\overline{1}, \overline{2}, \overline{0}).
\end{equation*}
Since $|A_{i_1,k} \cup B_{i_1,k}| \equiv 0 \Mod{3}$, the elements of $\overline{ g \cap (A_{i_1,k} \cup B_{i_1,k}) }$ are obtained via a concatenation of $3$-permutations.

We notice that $|A_{i_2,k}| \equiv 0 \Mod{3}$.

The sequence formed by the two greatest elements of $B_{i_2}$ and the smallest element of $A_{i_3}$ reads as follows (modulo $3$):
\begin{equation*}
(\overline{1}, \overline{2}, \overline{0}).
\end{equation*}
Since $|B_{i_2,k} \cup A_{i_3,k}| \equiv 0 \Mod{3}$, the elements of $\overline{ g \cap (B_{i_2,k} \cup A_{i_3,k}) }$ are obtained via a concatenation of $3$-permutations.

Hence we can say that the elements of $\overline{g \cap (\cup_{j=1}^3 (A_{i_j, k} \cup B_{i_j,k}))}$ are obtained by means of a concatenation of $3$-permutations and the same holds for $\overline{g \cap H_{2,k}}$.  
\end{enumerate}
\end{proof}

\begin{lemma}\label{family_3}
Let $S := \{a, b, c\}$, where $a:=6k+1$, $b:=9k+2$ and $c:=9k+3$ for some positive integer $k$, and
\begin{equation*}
H_{3,k} :=  (S + \{0, a \})  \cup (\cup_{i \in \N^*} T_{i,k}), 
\end{equation*} 
where
\begin{equation*}
T_{i,k} :=  A_{i,k} \cup B_{i,k} \cup C_{i,k} \cup D_{i,k} \cup E_{i,k} \cup F_{i,k} \cup G_{i,k} \cup I_{i,k} \cup J_{i,k}
\end{equation*}
with
\begin{align*}
A_{i,k} & := \{(3i) a \},\\
B_{i,k} & := [(3i) a  + 1, (3i) a + 3i],\\
C_{i,k} & := [(3i) a + 3k+1, (3i) a + 3k+2 + 3 (i-1)],\\
D_{i,k} & := A_{i,k} + \{ a \},\\
E_{i,k} & := B_{i,k} + \{ a \},\\
F_{i,k} & := [(3i) a + b, (3i) a + c + 3 i],\\ 
G_{i,k} & := D_{i,k} + \{ a \},\\
I_{i,k} & := E_{i,k} + \{ a \},\\
J_{i,k} & := F_{i,k} + \{ a \},
\end{align*}
for any $i \in \N^*$.

The following hold.
\begin{enumerate}
\item $H_{3,k}$ is a submonoid of $(\N, +,0)$ containing $S$.
\item For any $i \in [1,k-1]$ we have that
\begin{align*}
& A_{i,k} < B_{i,k} < C_{i,k} < D_{i,k} < E_{i,k} < F_{i,k} < G_{i,k} < I_{i,k} < J_{i,k},\\
& J_{i,k} < A_{i+1,k},
\end{align*}
and $A_{k,k} < B_{k,k} < C_{k,k} < D_{k,k}$.
\item $[(3k+1)a, \infty[ \subseteq H_{3,k}$.
\item $G = H_{3,k}$.
\item $H_{3,k}$ is a $3$-permutation semigroup.
\end{enumerate}
\end{lemma}

\begin{proof}
\begin{enumerate}[leftmargin=*]
\item By definition of $H_{3,k}$ we have that $S \subseteq H_{3,k}$.

Let $\{i_1, i_2 \} \subseteq \N^*$ and $i_3 := i_1 + i_2$.

If $x \in T_{i_1,k}$ and $y \in T_{i_2,k}$, then $x+y$ belongs to one of the rows 2-10, columns 2-10 of the following table. 

\begin{displaymath}
\begin{array}{|r|c|c|c|c|c|c|c|c|c|}
\hline
& A_{i_2,k} & B_{i_2,k} & C_{i_2,k} & D_{i_2,k} & E_{i_2,k} & F_{i_2,k} & G_{i_2,k} & I_{i_2,k} & J_{i_2,k}\\
\hline
A_{i_1,k} & A_{i_3,k} & B_{i_3,k} & C_{i_3,k} & D_{i_3,k} & E_{i_3,k} & F_{i_3,k} & G_{i_3,k} & I_{i_3,k} & J_{i_3,k}\\
\hline
B_{i_1,k} & & B_{i_3,k} & C_{i_3,k} & E_{i_3,k} & E_{i_3,k} & F_{i_3, k} & I_{i_3,k} & I_{i_3,k} & J_{i_3,k} \\
\hline
C_{i_1,k} & & & E_{i_3,k} & F_{i_3,k} & F_{i_3,k} & I_{i_3,k } & J_{i_3,k} & J_{i_3,k} & B_{i_3+1,k} \\
\hline
D_{i_1,k} & & & & G_{i_3,k} & I_{i_3,k} & J_{i_3,k} & A_{i_3+1,k}  & B_{i_3+1,k} & C_{i_3+1,k} \\
\hline
E_{i_1,k} & & & & & I_{i_3,k} & J_{i_3,k} & B_{i_3+1,k} & B_{i_3+1,k} & C_{i_3+1,k} \\
\hline
F_{i_1,k} & & & & & & B_{i_3+1,k} & C_{i_3+1,k} & C_{i_3+1,k} & E_{i_3+1,k} \\
\hline
G_{i_1,k} & & & & & & & D_{i_3+1,k} & E_{i_3+1,k} & F_{i_3+1,k} \\
\hline
I_{i_1,k} & & & & & & & & E_{i_3+1,k} & F_{i_3+1,k}  \\
\hline
J_{i_1,k} & & & & & & & & & I_{i_3+1,k}  \\
\hline
\end{array}
\end{displaymath}

\item All inequalities follow from the definition of the sets.

\item The assertion holds because
\begin{equation*}
D_{k,k} \cup E_{k,k} \cup F_{k,k} =[(3k) a + a, (3k) a+12k+3]
\end{equation*}
contains $6k+3$ consecutive integers. 

Therefore, if $x \in [(3k+1)a, \infty[$, then 
\begin{equation*}
x = y + ha
\end{equation*}
for some $y \in D_{k,k} \cup E_{k,k} \cup F_{k,k}$ and $h \in \N$. 

\item Let $x \in H_{3,k}$. 
\begin{itemize}
\item If $x \in (S+\{0, a \}) \cup A_{i,k} \cup D_{i,k} \cup G_{i,k}$ for some $i \in \N^*$, then we get immediately that $x \in G$.
\item Let $x \in B_{i,k}$ for some $i \in \N^*$. We notice that
\begin{equation*}
B_{i,k} = [(2i) b - i + 1, (2i) b + 2i].
\end{equation*}

If $x \in [(2i) b, (2i) b + 2i]$, then $x \in \langle \{b, c \} \rangle$ in accordance with Lemma \ref{lem_arith_seq}.

If $x = (2i) b - j$ for some $j$ with $1 \leq j \leq i-1$, then $x \in G$ because
\begin{equation*}
x = (2 i) b - j = 2 (i-j) b + (3 j) a. 
\end{equation*}

\item Let $x \in C_{i,k}$ for some $i \in \N^*$. We notice that $C_{i,k} = [t,u]$, where
\begin{align*}
t & := (2i-1) b + 2 a - (i-1),\\
u & := (2i-1) b + 2 a + (2i-1).
\end{align*}

If $x \in [(2i-1) b + 2 a, u ]$, then $x \in \langle \{ b, c \} \rangle + \{ 2a \}$ according to Lemma \ref{lem_arith_seq}. 

If $x = (2i-1) b + 2 a - j$, where $1 \leq j \leq i-1$, then $x \in G$ because 
\begin{equation*}
x = (2i-2j-1) b + 2 a + (3 j) a.
\end{equation*}

\item Let $x \in F_{i,k}$ for some $i \in \N^*$. We notice that
\begin{equation*}
F_{i,k} = [(2i+1)b - i, (2i+1)b+(2i+1)].
\end{equation*}

If $x \in [(2i+1) b, (2i+1)b+(2i+1)]$, then $x \in \langle \{ b, c \} \rangle $ according to Lemma \ref{lem_arith_seq}. 

If $x = (2i+1) b - j$, where $1 \leq j \leq i$, then $x \in G$ because 
\begin{equation*}
x = (2i-2j+1) b + (3 j) a.
\end{equation*}

\item If $x \in E_{i,k}  \cup I_{i,k} \cup J_{i,k}$, then $x \in G$ by the definition of the sets.
\end{itemize}

Hence we conclude that $H_{3,k} = G$.

\item Let $i \in [1,k]$. Then 
\begin{align*}
|A_{i,k}| & \equiv 1 \Mod{3},\\
|B_{i,k}| & \equiv 0 \Mod{3},\\
|C_{i,k}| & \equiv 2 \Mod{3},\\
|A_{i,k} \cup B_{i,k} \cup C_{i,k}| & \equiv 0 \Mod{3}.
\end{align*}

The sequence formed by the element of $A_{i,k}$ and the two smallest elements of $B_{i,k}$ reads (modulo $3$) as $(\overline{0},\overline{1},\overline{2})$.

The sequence formed by the greatest element of $B_{i,k}$ and the two smallest elements of $C_{i,k}$ reads (modulo $3$) as $(\overline{0},\overline{1},\overline{2})$.

The remaining elements of $\overline{g \cap C_{i,k}}$ are obtained through a concatenation of $3$-permutations. 

Therefore we conclude that the elements of $\overline{g \cap (A_{i,k} \cup B_{i,k} \cup C_{i,k})}$ can be written as a concatenation of $3$-permutations.

Since for any $i \in [1,k-1]$ we have that 
\begin{align*}
|D_{i,k} \cup E_{i,k} \cup F_{i,k}| = |G_{i,k} \cup I_{i,k} \cup J_{i,k}| & = 3 + 6i, 
\end{align*}
using a similar argument we can prove that the elements of 
\begin{equation*}
\overline{g \cap (D_{i,k} \cup E_{i,k} \cup F_{i,k} \cup G_{i,k} \cup I_{i,k} \cup J_{i,k})}
\end{equation*}
can be obtained via a concatenation of $3$-permutations.

Hence $H_{3,k}$ is a $3$-permutation semigroup.
\end{enumerate}
\end{proof}

\begin{lemma}\label{family_4}
Let $S := \{a, b, c \}$, where 
\begin{equation*}
a := 6k+1, \quad b:= a + 2, \quad c := a +4,
\end{equation*}
for some positive integer $k$, and
\begin{displaymath}
t  := \left\lfloor \frac{3}{2} k \right\rfloor, \quad 
\epsilon  := 
\begin{cases}
0 & \text{if $k$ is even},\\
1 & \text{if $k$ is odd}.
\end{cases}
\end{displaymath}
Let 
\begin{equation*}
H_{4,k} := \left( \cup_{i=1}^{t} A_{i,k} \right) \cup \left( \cup_{i=t+1}^{2t+\epsilon} (B_{i,k} \cup C_{i,k}) \right) \cup [(2t+\epsilon+1)a, \infty[,
\end{equation*}
where
\begin{align*}
A_{i,k} & := [ia, ia + 4i]_2
\end{align*}
if $i \in [1,t]$, while
\begin{align*}
B_{i,k} & := [ia, ia + 2 (2 (i-t-1) - \epsilon)], \\
C_{i,k} & := [ia + 2 (2 (i-t-1) - \epsilon +1), ia + 2 (2t + \epsilon)]_2,
\end{align*}
if $i \in [t+1,2t+ \epsilon]$.

The following hold.
\begin{enumerate}
\item $H_{4,k}$ is a submonoid of $(\N, +, 0)$ containing $S$.
\item The following inequalities hold.
\begin{itemize}
\item $A_{i,k} < A_{i+1,k}$ for any $i \in[1,t-1]$.
\item $A_{t,k} < B_{t+1,k}$ if $k$ is even, while $A_{t,k} < C_{t+1,k}$ and $B_{t+1,k} = \emptyset$ if $k$ is odd.
\item $B_{t+1,k} < C_{t+1,k}$ if $k$ is even.
\item $B_{i,k} < C_{i,k}$ if $i \in [t+2, 2t+ \epsilon]$.
\item $C_{i,k} < B_{i+1,k}$ if $i \in [t+1, 2t+ \epsilon-1]$.
\item $C_{2t+\epsilon,k} < [(2t+\epsilon+1)a, \infty[$.
\end{itemize}
\item $G = H_{4,k}$.
\item $H_{4,k}$ is a $3$-permutation semigroup.
\end{enumerate}
\end{lemma}
\begin{proof}
\begin{enumerate}[leftmargin=*]
\item We notice that $S = A_{1,k}$.

Let $x \in A_{i_1,k}$ and $y \in A_{i_2,k}$ for some $\{i_1, i_2 \} \subseteq [1,t]$. Then 
\begin{align*}
x & = i_1 a + 2 h_1 \quad \text{with $h_1 \in [0, 2i_1]$,}\\
y & = i_2 a + 2 h_2 \quad \text{with $h_2 \in [0, 2 i_2]$.}
\end{align*}

\begin{itemize}
\item If $i_1 + i_2 \leq t$, then $x + y \in A_{i_1+i_2,k}$.

\item If $t < i_1 + i_2 \leq 2 t + \epsilon$ and $h_1 + h_2 \leq 2t + \epsilon$, then
\begin{equation*}
x + y \in B_{i_{1} + i_2, k} \cup C_{i_1 + i_2, k}.
\end{equation*}

\item If $t < i_1 + i_2 < 2 t + \epsilon$ and $h_1 + h_2 > 2t + \epsilon$, then
\begin{equation*}
2 (h_1 + h_2) > 4 t + \epsilon = 6 k,
\end{equation*}
namely $2 (h_1 + h_2) \geq a$.

Since
\begin{equation*}
x + y = (i_1+i_2+1) a + 2 (h_1+h_2) - a, 
\end{equation*}
where
\begin{equation*}
2 (h_1+h_2) - a \leq 4 (i_1 + i_2) - (4t + 2 \epsilon) = 4 (i_1 + i_2 + 1 - t - 1) - 2 \epsilon,
\end{equation*}
we conclude that 
\begin{equation*}
x + y \in B_{i_1 + i_2 + 1, k}.
\end{equation*}

\item If $i_1+i_2 > 2 t + \epsilon$ or $i_1 + i_2 = 2 t + \epsilon$ and $h_1 + h_2 > 2t + \epsilon$, then 
\begin{equation*}
x + y \in [(2t+\epsilon+1) a, \infty[.
\end{equation*}
\end{itemize}

\item All inequalities follow from the definition of the sets.

\item Let $x \in H_{4,k}$.
\begin{itemize}
\item If $x \in A_{i,k}$ for some $i \in [1,t]$, then $x \in G$ according to Lemma \ref{lem_arith_seq}.  
\item If $x \in C_{i,k}$ for some $i \in [t+1,2t + \epsilon]$, then $x \in G$ according to Lemma \ref{lem_arith_seq}. 
\item If $x \in [(2t+\epsilon+1)a, \infty[$, then $x \in G$. 

Indeed, according to Lemma \ref{lem_arith_seq} we have that
\begin{align*}
F(G) & = \left\lfloor \frac{6k-1}{2} \right\rfloor a + 2 (6k) = \left\lfloor \frac{4t-1+2 \epsilon}{2} \right\rfloor a + 2 (6k)\\
& = (2t+ \epsilon-1) a + 2 (6k) = (2 t + \epsilon + 1) a -2.
\end{align*}

\item Let $x \in B_{i,k}$ for some $i \in [t+1, 2t + \epsilon]$. 

If $x = i a + 2 h$ for some $h \in [0, 2(i-t-1) - \epsilon]$, then $h \leq 2i$ and $x \in G$ according to Lemma \ref{lem_arith_seq}.

If $x = i a + \delta$ for some odd integer $\delta \in [0, 4(i-t-1) - 2 \epsilon]$, then 
\begin{equation*}
x = (i - 1) a + \delta + a, 
\end{equation*} 
where
\begin{equation*}
\delta \leq 4(i-t-1) - 2 \epsilon - 1.
\end{equation*}
We notice that
\begin{equation*}
\frac{\delta + a}{2} = 2(i-t-1) -  \epsilon + 3k.
\end{equation*}
Since $3k-2t-\epsilon \leq 0$, we deduce that
\begin{equation*}
\frac{\delta + a}{2} \leq 2 (i-1)
\end{equation*}
and $x \in G$ in accordance with Lemma \ref{lem_arith_seq}.
\end{itemize}
Therefore we conclude that $G = H_{4,k}$.

\item If $k$ is even, then $t \equiv 0 \Mod{3}$. Let $i_1, i_2$ and $i_3$ be three consecutive positive integers such that $i_j \equiv j \Mod{3}$ for $j \in [1,3]$. Then
\begin{align*}
|A_{i_1,k}| \equiv 0 \Mod{3},\\
|A_{i_2,k}| \equiv 2 \Mod{3},\\
|A_{i_3,k}| \equiv 1 \Mod{3},
\end{align*}
and
\begin{equation*}
|A_{i_1,k} \cup A_{i_2,k} \cup A_{i_3,k}| \equiv 0 \Mod{3}.
\end{equation*}

The sequence formed by the two greatest elements of $A_{i_2,k}$ and the smallest element of $A_{i_3,k}$ reads (modulo $3$) as $(\overline{2}, \overline{1}, \overline{0})$.

The remaining elements of $\overline{g \cap (A_{i_1,k} \cup A_{i_2,k} \cup A_{i_3,k})}$ are obtained through concatenations of $3$-permutations.

For any $i \in [t+1, 2t]$ we have that 
\begin{align*}
|B_{i,k}| & \equiv i \Mod{3},\\
|C_{i,k}| & \equiv i - 1 \Mod{3},\\
|B_{i,k} \cup C_{i,k}| & \equiv 2 i - 1 \Mod{3}.
\end{align*}

Moreover $|[t+1, 2t]| = t \equiv 0 \Mod{3}$. 

Let $i_1, i_2$ and $i_3$ be three consecutive positive integers with $i_j \equiv j \Mod{3}$ for $j \in [1,3]$.

If $x_{i_1}$ is the greatest element of $B_{i_1, k}$ and $y_{i_1} < z_{i_1}$ are the smallest elements of $C_{i_1,k}$, then
\begin{equation*}
(\overline{x_{i_1}}, \overline{y_{i_1}}, \overline{z_{i_1}}) = (\overline{1}, \overline{0}, \overline{2}).
\end{equation*}

We have that $|C_{i_1,k} \backslash \{y_{i_1}, z_{i_1} \}| \equiv 1 \Mod{3}$. If $x_{i_2}$ is the greatest element of $C_{i_1, k}$ and $y_{i_2} < z_{i_2}$ are the smallest elements of $B_{i_2,k}$, then
\begin{equation*}
(\overline{x_{i_2}}, \overline{y_{i_2}}, \overline{z_{i_2}}) = (\overline{1}, \overline{2}, \overline{0}).
\end{equation*}

Since $|B_{i_2,k} \backslash \{ y_{i_2}, z_{i_2} \}| \equiv 0 \Mod{3}$, we can concentrate on $C_{i_2,k}$.

If $x_{i_3}$ is the greatest element of $C_{i_2, k}$ and $y_{i_3} < z_{i_3}$ are the smallest elements of $B_{i_3,k}$, then
\begin{equation*}
(\overline{x_{i_3}}, \overline{y_{i_3}}, \overline{z_{i_3}}) = (\overline{2}, \overline{0}, \overline{1}).
\end{equation*}

We have that $|B_{i_3,k} \backslash \{y_{i_3}, z_{i_3} \}| \equiv 1 \Mod{3}$. If $x_{i_4}$ is the greatest element of $B_{i_3, k}$ and $y_{i_4} < z_{i_4}$ are the smallest elements of $C_{i_3,k}$, then
\begin{equation*}
(\overline{x_{i_4}}, \overline{y_{i_4}}, \overline{z_{i_4}}) = (\overline{2}, \overline{1}, \overline{0}).
\end{equation*}

We conclude that $H_{4,k}$ is a $3$-permutation semigroup when $k$ is even.

Using similar arguments we can prove that $H_{4,k}$ is a $3$-permutation semigroup when $k$ is odd. 
\end{enumerate}

\end{proof}

\begin{lemma}\label{family_5}
Let $S := \{a, b, c \}$, where
\begin{equation*}
a := 6k+1, \quad b:= 2a - 6, \quad c:= 2a -2
\end{equation*}
for some integer $k \geq 2$.

Let
\begin{equation*}
H_{5,k} := \left( \cup_{i=0}^{k-1} (A_{i,k} \cup B_{i,k}) \right) \cup D_{k,k} \cup \left( \cup_{i=k+1}^{2k-1} (C_{i,k} \cup D_{i,k}) \right) \cup [(4k-1) a + 5, \infty[,
\end{equation*}
where
\begin{displaymath}
\begin{array}{rcl}
A_{i,k} & := & \{(2i+1)a, (2i+2)a - 6 (i+1), (2i+2) a - 6 (i+1) +4 \} \\
& & \cup [ (2i+2) a - 6 i, (2i+2) a - 2 ]_2,\\
\\
B_{i,k} & := & \{(2i+2)a, (2i+3)a - 6 (i+1), (2i+3) a - 6 (i+1) +4 \} \\
& & \cup [ (2i+3) a - 6 i, (2i+3)a - 2 ]_2,
\end{array}
\end{displaymath}
for $i \in [0,k-2]$,
\begin{displaymath}
\begin{array}{rcl}
A_{k-1,k} & := & \{(2k-1)a, (2k) a - 6 k, (2k) a - 6 k + 4 \} \\
& & \cup [ (2k) a - 6 k + 6, (2k) a - 2 ]_2,\\
\\
B_{k-1,k} & := & \{(2k) a, (2k+1)a - 6 k, (2k+1) a - 6 k +4 \} \\
& & \cup [ (2k+1) a - 6 k + 6, (2k+1) a - 8 ]_2\\
& & \cup [ (2k+1) a - 6, (2k+1) a - 4 ] \cup [(2k+1) a - 2, (2k+1) a],\\
\\
D_{k,k} & := &  [ (2k+1) a + 1, (2k+1) a + 1 + 6k-8]_2\\
& & \cup [ (2k+2) a - 6, (2k+2) a - 4 ] \cup [(2k+2) a - 2, (2k+2) a],
\end{array}
\end{displaymath}
and
\begin{displaymath}
\begin{array}{rcl}
C_{i,k} & := & [(2i) a + 1, (2i) a + 1 + 2 (6k-3i-4) ]_2 \\
& & \cup [ (2 i) a + 1 + 2 (6k -3i - 3) ,  (2 i) a + 3 + 2 (6k -3i - 3) ]\\
& & \cup [ (2 i) a + 5 + 2 (6k-3i-3), (2i+1) a ],\\
\\
D_{i,k} & := & [(2i+1) a + 1, (2i+1) a + 1 + 2 (6k-3i-4) ]_2 \\
& & \cup [ (2i+1) a + 1 + 2 (6k -3i - 3) ,  (2i+1) a + 3 + 2 (6k -3i - 3) ]\\
& & \cup [(2i+1) a + 5 + 2 (6k-3i-3), (2i+2) a ],
\end{array}
\end{displaymath}
for $i \in [k+1,2k-1]$.

The following hold.
\begin{enumerate}
\item We have that
\begin{align*}
C_{2k-1,k} & = [(4k-2) a + 1, (4k-2) a + 3] \cup [(4k-2) a +5, (4k-1)a],\\
D_{2k-1,k} & = [(4k-1) a + 1, (4k-1) a + 3] \cup [(4k-1) a +5, (4k)a].
\end{align*}
\item $x \in G$ if and only if 
\begin{equation*}
x = (p+2q) a + 4 r - 6 q,
\end{equation*}
where $\{p, q, r \} \subseteq \N$ and $0 \leq r \leq q$.
\item $H_{5,k} \subseteq G$.
\item $G \subseteq  H_{5,k}$.
\item The following inequalities hold:
\begin{displaymath}
\begin{array}{ll}
A_{i,k} < B_{i,k} & \text{for any $i \in [0,k-1]$;}\\
B_{i,k} < A_{i+1,k} & \text{for any $i \in [0,k-2]$;}\\
B_{k-1,k} < D_{k,k};\\
C_{i,k} < D_{i,k} & \text{for any $i \in [k+1,2k-1]$;}\\
D_{i,k} < C_{i+1,k} & \text{for any $i \in [k+1,2k-2]$.}
\end{array}
\end{displaymath}
\item $H_{5,k}$ is a $3$-permutation semigroup.
\end{enumerate}
\end{lemma}

\begin{proof}
\begin{enumerate}[leftmargin=*]
\item Both the assertions follow from the definition of the sets.
\item We have that $x \in G$ if and only if 
\begin{equation*}
x = p a + s b + t c
\end{equation*}
for some $\{p, s, t \} \subseteq \N$.

Since $s b + t c \in \langle b, c \rangle$, according to Lemma \ref{lem_arith_seq} we can write
\begin{align*}
x & = p a + q b + 4 r = (p + 2q) a + 4r-6q,
\end{align*}
where $0 \leq r \leq q$.

\item We prove the assertion dealing with different cases.
\begin{itemize}
\item \emph{Case 1:} $x \in A_{i,k} \cup B_{i,k}$ with $i \in [0,k-2]$. 

Since the case $x \in B_{i,k}$ is very similar to the case $x \in A_{i,k}$, we discuss in detail just this latter one.

If $x \in \{(2i+1)a, (2i+2)a - 6 (i+1), (2i+2) a - 6 (i+1) +4 \}$, then $x \in G$ by the characterization of the elements belonging to $G$.

Now suppose that $x = (2i+2) a - 6 i +2 h$ with $h \in [0,3i-1]$. We notice that $-6i + 2h$ can be written in the form 
\begin{equation*}
-6 \tilde{i} + 4 \tilde{h},
\end{equation*}
where $\tilde{i} \in [1, i+1]$ and $\tilde{h} \in [0,2]$. Hence
\begin{equation*}
x = (2 + 2 (i - \tilde{i}) + 2 \tilde{i}) a + 4 \tilde{h} - 6 \tilde{i} ,
\end{equation*}
where $\tilde{i}$ and $\tilde{h}$ are as above.

We have that $\tilde{h} \leq \tilde{i}$ in the case that $\tilde{i} \geq 2$. 

If $\tilde{i} = 1$, then $\tilde{h} \leq 1$. In fact, if $\tilde{h} = 2$, then $x > (2i+2) a$ in contradiction with the fact that $x < (2i+2) a$.

\item \emph{Case 2:} $x \in A_{k-1,k} \cup B_{k-1,k}$. 

The proof follows the same lines as in Case 1. We discuss in detail just the case
\begin{equation*}
x \in [ (2k+1) a - 6, (2k+1) a - 4 ] \cup [(2k+1) a - 2, (2k+1) a].
\end{equation*}
More precisely we show that any $x$ in such a range belongs to $G$. 

Obviously we have that $(2k+1) a \in G$. Therefore we deal with the remaining elements.
\begin{itemize}
\item $(2k+1) a - 6 = (1 + 2 (k-1) + 2 \cdot 1) a - 6 \cdot 1$.
\item $(2k+1) a - 5 = (2(k+1)) a - 6 (k+1)$.
\item $(2k+1) a - 4 = (1 + 2 (k-2) + 2 \cdot 2) a + 4 \cdot 2 - 6 \cdot 2$.
\item $(2k+1) a - 2 = (1 + 2 (k-1) + 2 \cdot 1) a + 4 \cdot 1 - 6 \cdot 1$.
\item $(2k+1) a - 1 = (2(k+1)) a + 4 \cdot 1 - 6 (k+1)$.
\end{itemize}

\item \emph{Case 3:} $x \in C_{i,k} \cup D_{i,k}$ with $i \in[k+1,2k-1]$ or $x \in D_{k,k}$. 

We discuss in detail just the case $x \in C_{i,k}$.

If $x \in C_{i,k}$ and $x$ is odd, then 
\begin{align*}
x & = (2i) a + 1 + 2 h \\
& = (2i+1) a - 6 k + 2 h
\end{align*}
for some $h \in [0,3k]$.

We can prove that $x \in G$ as in Case 1. More precisely,
\begin{equation*}
-6 k + 2 h = - 6 \tilde{k} + 4 \tilde{h}
\end{equation*}
for some $\tilde{k} \in [1,k+1]$ and $\tilde{h} \in [0,2]$.

Since
\begin{align*}
(2i+1) a - 6 k + 2 h & = (2i+1) a - 6 \tilde{k} + 4 \tilde{h} \\
& = (1 + 2 (i - \tilde{k}) + 2 \tilde{k})  + 4 \tilde{h} - 6 \tilde{k},
\end{align*} 
we conclude that $x \in G$.

If $x \in C_{i,k}$ and $x$ is even, then 
\begin{align*}
x = (2i)a + 2 + 2 (6k-3i-3) \quad \text{ or } \quad x \geq (2i)a + 6 + 2 (6k-3i-3).
\end{align*}
In the former case we notice that
\begin{align*}
(2i)a + 2 + 2 (6k-3i-3) & = (2i) a - 6 + 2 (6k+1) - 6 i \\
& = (2i+2) a - 6 (i+1).
\end{align*}

In the latter case, if $x = (2i) a + 6 + 2 (6k-3i-3)$, then 
\begin{align*}
x & = (2 + 2 i) a - 6i. 
\end{align*}

In the remaining cases we have that 
\begin{align*}
x = (2i)a + 8 + 2 (6k-3i-3) + 2 h,
\end{align*}
namely
\begin{align*}
x & =  (2i+2) a + 6 - 6 (i+1) + 2 h\\
& = (2i+2) a - 6 i + 2 h
\end{align*}
for some $h \geq 0$. Hence the assertion can be proved as in Case 1.

\item \emph{Case 4:} $x \in [(4k-1) a + 5, \infty[$. 

By the characterization of the elements of $D_{2k-1,k}$ we have that
\begin{equation*}
[(4k-1) a + 5,  (4k)a] \subseteq G
\end{equation*}
and
\begin{equation*}
[(4k) a + 1, (4k) a + 3 ] = [(4k-1) a + 1, (4k-1) a + 3 ] + \{ a \} \subseteq G.
\end{equation*}
Since 
\begin{equation*}
4k a + 4 = (4k-2) a + 6 + 12 k
\end{equation*}
and $(4k-2)a + 6 \in C_{2k-1,k}$, while $12 k \in A_{0,k}$, we conclude that $\Ap(G, a) \subseteq [a, 4k a + 4]$. Hence we conclude that $[(4k-1) a + 5, \infty[ \subseteq G$.

\item Let $x \in G$. Then 
\begin{equation*}
x = (p + 2q) a + 4 r - 6 q,
\end{equation*} 
where $\{p, q, r \} \subseteq \N$ and $0 \leq r \leq q$.

Let $j:= p + 2 q$.

First we consider the case $j \geq 4k + 2$. 

Since $q \leq \frac{j}{2}$, we have that
\begin{align*}
j a + 4 r - 6 q & \geq j (a-3) \geq (4k+2) (a-3) \\
& = (4k+2) a - 2 (6k+3) = (4 k +2) a - 2 a - 4\\
& = (4k-1) a + a - 4\\
& \geq (4k-1) a + 9
\end{align*}
because $a \geq 13$. Hence $x \in G$.

Now we consider the case $j \leq 4k+1$. We notice that
\begin{equation*}
j a - 3 j = j (a-3) \geq (4k) (a-3) > (4k -2) a. 
\end{equation*}

We discuss separately some subcases.
\begin{itemize}
\item \emph{Subcase 1:} $j = 2 i +2 $ for some $i \in [0,k-2]$. 

We have that $x \in A_{i,k} \cup B_{i,k}$ because
\begin{align*}
x = (2i+2) a + 4 r - 6 q 
\end{align*}
with $q \leq i+1$, namely 
\begin{displaymath}
\begin{array}{c}
x \in \{j a - 6 (i+1), j a - 6 (i+1) +4 \} \cup \{j a \} \\
\text{or}\\
x = j a - 6 i + 2 h 
\end{array}
\end{displaymath}

for some $h \in [0,3i-1]$.

\item \emph{Subcase 2:} $j = 2 i + 3 $ for some $i \in [0,k-2]$. 

We can prove as above that $x \in B_{i,k} \cup A_{i+1,k}$.

\item \emph{Subcase 3:} $j \in \{2 k, 2 k +1 \}$. 

We can prove as above that $x \in A_{k-1,k} \cup B_{k-1,k}$.

\item \emph{Subcase 4:} $j = 2 i + 2$ for some $i \in [k+1, 2k-1]$.

If $q \leq k$, then we can prove as above that $x \in D_{i,k}$.

If $q > k$, then $x \in C_{i,k}$. In fact, $(2i+1) a   > x$ and
\begin{align*}
x & = (2i+2) a + 4 r - 6 q \\
& \geq (2i+2) a - 6 q \geq (2 i) a + 2 (6 k +1) - 6 (i +1)\\
& = (2i) a + 2 + 2 (6k -3 i -3).
\end{align*}
In particular, 
\begin{equation*}
x = (2i) a + 2 + 2 (6k -3 i -3)
\end{equation*}
if $r = 0$ and $q = i +1$, while 
\begin{equation*}
x \geq (2i) a + 6 + 2 (6k -3 i -3)
\end{equation*}
in the other cases.

\item \emph{Subcase 5:} $j = 2 i + 1$ for some $i \in [k+1, 2k]$.

We can prove as above that $x \in C_{i,k} \cup D_{i-1,k}$.
\end{itemize}
\end{itemize}

\item All inequalities follow from the definitions of the sets.

\item First we notice that
\begin{align*}
|A_{i,k}| = |B_{i,k}| = |C_{i,k}| = |D_{i,k}| = 3 i +3
\end{align*}
for any suitable $i$, while
\begin{align*}
|A_{k-1,k}| & = 3 k, \\
|B_{k-1,k}| = |D_{k,k}| & = 3 k + 3.\\
\end{align*}
We notice that $3$ divides the cardinality of all the sets $A_{i,k}, B_{i,k}, C_{i,k}, D_{i,k}$ and the elements of $g$ in such sets are obtained (modulo $3$) via a concatenation of $3$-permutations. We check that this latter assertion is true for the sets $A_{i,k}$. The assertion can be verified in a similar way for the remaining sets.

For a fixed index $i$ we have that
\begin{align*}
(2i+1) a & \equiv 2 i + 1 \Mod{3},\\
(2i+2) a - 6 (i+1) & \equiv 2 i + 2 \Mod{3},\\
(2i+2) a - 6 (i+1) + 4 & \equiv 2 i \Mod{3}, 
\end{align*}
while the elements in 
\begin{equation*}
[ (2i+2) a - 6 i, (2i+2) a - 2 ]_2
\end{equation*}
are given (modulo $3$) by repeated concatenations of the sequence 
\begin{equation*}
(\overline{2i+2}, \overline{2i+2+2}, \overline{2i+2+1}).
\end{equation*}
\end{enumerate}
\end{proof}

\begin{lemma}\label{family_6}
Let $k$ be positive integer and $a:=3k+1$.

Let $S := \{ a, 2a - 3, 2 a - 2 \}$ and $H_{6,k} :=\langle S \rangle$.

Then $H_{6,k}$ is a $3$-permutation semigroup.
\end{lemma}
\begin{proof}
See Lemma \ref{family_n_1}.
\end{proof}

\begin{lemma}\label{family_7}
Let $k$ be a positive integer and $a:=3k+2$.

Let $S := \{a, a+1, a+2 \}$, $t := \lfloor \frac{3k}{2} \rfloor$ and

\begin{equation*}
H_{7,k} := \cup_{i=0}^t I_{i,k} \cup [(t+1)a, \infty[,
\end{equation*}
where 
\begin{equation*}
I_{i,k} := [ia, ia + 2 i]
\end{equation*}
for any $i \in [0,t+1]$.

Then the following hold.

\begin{enumerate}
\item $G = H_{7,k}$.
\item $I_{i,k} < I_{i+1,k}$ for any $i \in [0,t]$ and $I_{t,k} < [(t+1)a, \infty[$.
\item $H_{7,k}$ is a $3$-permutation semigroup.
\end{enumerate}
\end{lemma}

\begin{proof}
\begin{enumerate}[leftmargin=*]
\item Let $x \in I_{i,k}$ for some $i \in [0,t]$. Then $x \in G$ according to Lemma \ref{lem_arith_seq}. As claimed in the same lemma, we have that 
\begin{align*}
F(G) & = (3k+2) \left\lfloor \frac{3k}{2} \right\rfloor + (3k+1)\\
& = (3k+2) t + (3k+1) < (t+1) a.
\end{align*}
Therefore $[(t+1)a, \infty[ \subseteq G$.

Let $x \in G$. 

If $x < (t+1) a$, then $x \in I_{i,k}$ for some $i$, according to Lemma \ref{lem_arith_seq}. 

If $x \geq (t+1) a$, then $x \in [(t+1) a, \infty[$.

Hence $G \subseteq H_{7,k}$.

\item All inequalities can be easily verified. 

\item We notice that $H_{7,k} \backslash \{ 0 \} = (\cup_{i=1}^t I_{i,k}) \cup [(t+1)a, \infty[$. 

If $i_1, i_2$ and $i_3$ are three consecutive positive integers such that $i_j \equiv j \Mod{3}$ for any $j \in \{1,2,3 \}$, then
\begin{align*}
|I_{i_1,k}| & \equiv 0 \Mod{3},\\
|I_{i_2,k}| & \equiv 2 \Mod{3},\\
|I_{i_3,k}| & \equiv 1 \Mod{3}.
\end{align*}

Since $3$ divides $|I_{i_1,k}|$, we concentrate on the elements of $I_{i_2}$ and $I_{i_3}$. 

Let $x_{i_2}$ and $y_{i_2}$ be the greatest elements of $I_{i_2,k}$ and $z_{i_2}$ the smallest element of $I_{i_3,k}$. Then
\begin{equation*}
(\overline{x_{i_2}}, \overline{y_{i_2}}, \overline{z_{i_2}}) = (\overline{1}, \overline{2}, \overline{0}).
\end{equation*} 
The remaining elements of $I_{i_1}, I_{i_2}$ and $I_{i_3}$ can be obtained (modulo $3$) through concatenations of $3$-permutations. Therefore we  conclude that $H_{7,k}$ is a $3$-permutation semigroup. 
\end{enumerate}
\end{proof}

\begin{lemma}\label{family_8}
Let $k$ be a positive integer and $a:=12k+2$.

Let $S := \left\{a, a+2, a +  \frac{a}{2} \right\}$ and
\begin{equation*}
H_{8,k} := A_{0,k} \cup \left( \cup_{i=1}^{3k} (A_{i,k} \cup B_{i,k}) \right) \cup \left(\cup_{i=3k+1}^{6k+1} C_{i,k} \right) \cup \left(\cup_{i={3k+1}}^{6k} D_{i,k} \right) \cup E,
\end{equation*}
where $A_{0,k} := \{ 0 \}$, while
\begin{align*}
A_{i,k} & := [ia, ia  + 2 i ]_2,\\
B_{i,k} & := A_{i-1,k} + \left\{ a + \frac{a}{2}  \right\},
\end{align*}
for any $i \in [1,3k]$, 
\begin{align*}
D_{i,k} & := \left[ i a + \frac{a}{2}, ia + \frac{a}{2} + 2 (i-3k) \right] \cup \left[ i a + \frac{a}{2} + 2 (i-3k+1), i a + \frac{a}{2} + 6 k  \right]_2
\end{align*}
for any $i \in [3k+1,6k]$,
\begin{align*}
C_{3k+1,k} & := [(3k+1) a, (3k+1)a + 6 k ]_2,\\
C_{3k+2,k} & := [(3k+2) a, (3k+2)a + 6 k ]_2,
\end{align*}
while
\begin{align*}
C_{i,k} & := D_{i-2,k} + \left\{ a + \frac{a}{2} \right\}
\end{align*}
for any $i \in [3k+1,6k+1]$ and
\begin{equation*}
E:=\left[ (6k+1) a + \frac{a}{2}, \infty \right[.
\end{equation*}

Then the following hold.

\begin{enumerate}
\item $x \in G$ if and only if 
\begin{equation*}
x = q a + 2 r + s \left( a + \frac{a}{2} \right),
\end{equation*}
where $\{q, r, s \} \subseteq \N$ and $0 \leq r \leq q$.
\item $G = H_{8,k}$.
\item The following inequalities hold:
\begin{displaymath}
\begin{array}{ll}
A_{i,k} < B_{i,k} & \text{for any $i \in [1,3k]$;}\\
B_{i,k} < A_{i+1,k} & \text{for any $i \in [1, 3k-1]$;}\\
B_{3k,k} < C_{3k+1,k}; \\
C_{i,k} < D_{i,k} & \text{for any $i \in [3k+1,6k]$;}\\
D_{i,k} < C_{i+1,k} & \text{for any $i \in [3k+1,6k]$;}\\
C_{6k+1,k} < E. 
\end{array}
\end{displaymath}
\item $H_{8,k}$ is a $3$-permutation semigroup.
\end{enumerate}
\end{lemma}
\begin{proof}
\begin{enumerate}[leftmargin=*]
\item By definition, $x \in G$ if and only if 
\begin{equation*}
x = t a + u (a+2) + v \left(a + \frac{a}{2} \right)
\end{equation*}
for some $\{t, u, v \} \subseteq \N$. Since $ta + u (a+2) \in \langle a, a+2 \rangle$, according to Lemma \ref{lem_arith_seq} we can say that $x \in G$ if and only if 
\begin{equation*}
x = q a + 2 r + v \left( a + \frac{a}{2} \right),
\end{equation*}
where $\{q, r \} \subseteq \N$ and $0 \leq r \leq q$.

\item Let $x \in H_{8,k}$. 

If $x$ belongs to one of the sets $A_{i,k}, B_{i,k}, C_{i,k}$ or $D_{i,k}$ for some index $i$, then $x \in G$ in virtue of the way such sets have been defined and the characterization of the elements of $G$.

Moreover, we notice that
\begin{equation*}
D_{6k,k} \cup C_{6k+1,k} = \left[ 6 k a + \frac{a}{2}, (6k+1)a + 6 k \right] \backslash \{(6k+1) a + 6k -1 \}
\end{equation*}
and
\begin{equation*}
(6k+2) a + 6k -1 = (6k+1) a + 6k -3 + (a+2) \subseteq C_{6k+1,k} + S.
\end{equation*}
Hence $F(G) = (6k+1) a + 6k -1$ and $E \subseteq G$.

Vice versa, if $x \in G$ and $x < (6k+1) a + \frac{a}{2}$, then 
\begin{equation*}
x = (s+q) a + 2 r + s \cdot \frac{a}{2},
\end{equation*}
where $s = 2 t + \epsilon$ with $\epsilon \in [0,1]$. Therefore we can write that
\begin{equation*}
x = i a + 2 r + \epsilon \cdot \frac{a}{2},
\end{equation*}
where $ i:= t + q$.

If $\epsilon = 0$, then $x \in A_{i,k} \cup C_{i,k}$, while $x \in B_{i,k} \cup D_{i,k}$ if $\epsilon = 1$.

\item All inequalities can be easily verified.

\item If $i_1, i_2$ and $i_3$ are three consecutive positive integers such that $i_j \equiv j \Mod{3}$ for any $j \in \{1 ,2, 3 \}$, then
\begin{equation*}
| \cup_{j=1}^3 (A_{i_j} \cup B_{i_j}) | \equiv 0 \Mod{3}.
\end{equation*} 

For any $j \in \{1 ,2, 3 \}$ we have that
\begin{align*}
|A_{i_j} | & = i_j + 1  \equiv j +1 \Mod{3},\\
|B_{i_j} | & = i_j  \equiv j \Mod{3}.
\end{align*}

If $j = 1$, then $|A_{i_1} \cup B_{i_1}| \equiv 0 \Mod{3}$. Let $x_{i_1}$ and $y_{i_1}$ be the greatest elements of $A_{i_1}$ and $z_{i_1}$ the smallest element of $B_{i_1}$. Then 
\begin{equation*}
(\overline{x_{i_1}}, \overline{y_{i_1}}, \overline{z_{i_1}}) = (\overline{2}, \overline{1}, \overline{0}).
\end{equation*} 
We notice in passing that the other elements of $\overline{g \cap A_{i_1}}$ and $\overline{g \cap B_{i_1}}$ are obtained through concatenations of $3$-permutations.
Using a similar argument we can prove that the elements of $\overline{g \cap (\cup_{j=2}^3 (A_{i_j} \cup B_{i_j}))}$ are obtained through a concatenation of $3$-permutations.

As regards the sets $C_{i,k}$ and $D_{i,k}$, first we notice that
\begin{align*}
|C_{3k+1,k}| & = 3 k +1,\\
|D_{3k+1,k}| & = 3 k +2.
\end{align*}
The greatest element of $C_{3k+1,k}$ and the two smallest elements of $D_{3k+1,k}$ form (modulo $3$) a $3$-permutation, while the remaining elements of $\overline{g \cap C_{3k+1,k}}$ and $\overline{g \cap D_{3k+1,k}}$ are given by concatenations of $3$-permutations.

In general we notice that 
\begin{align*}
|C_{i,k}| & = i - 1,\\
|D_{i,k}| & = i + 1. 
\end{align*}
Hence, if we take  three consecutive positive integers $i_1, i_2$ and $i_3$ such that $i_j \equiv j \Mod{3}$ for any $j$, then we can check as above that the sequence of elements $\overline{g \cap (\cup_{j=1}^3 (C_{i_j,k} \cup D_{i_j,k}))}$ is given by a concatenation of $3$-permutations.

\end{enumerate}
\end{proof}

\begin{lemma}\label{family_9}
Let $k$ be a positive integer, $a:=3k+2$, $b:=2a-3$ and $c:=2a-1$.

Let $S:=\{a, b, c \}$ and
\begin{equation*}
H_{9,k} := \left( \cup_{i=0}^{2k+1} I_{i,k} \right) \cup [2ka + 4, \infty[,
\end{equation*}
where 
\begin{align*}
I_{i,k} & := \left\{ia - 3 \left\lfloor \frac{i}{2} \right\rfloor \right\} \cup \left[ i a - 3 \left\lfloor \frac{i}{2} \right\rfloor +2, i a \right]
\end{align*}  
for any $i \in [0,2k+1]$.

Then the following hold.
\begin{enumerate}
\item $x \in G$ if and only if 
\begin{equation*}
x = (s+2q) a - 3 q + 2 r,
\end{equation*}
where $\{q, r, s \} \subseteq \N$ with $0 \leq r \leq q$.
\item We have that
\begin{displaymath}
\begin{array}{ll}
I_{i,k} < I_{i+1,k} & \text{for any $i \in [0,2k]$.}\\
\end{array}
\end{displaymath}
\item $G = H_{9,k}$.
\item $H_{9,k}$ is a $3$-permutation semigroup.
\end{enumerate}
\end{lemma}

\begin{proof}
\begin{enumerate}[leftmargin=*]
\item The assertion follows from Lemma \ref{lem_arith_seq}.
\item The inequality follows from the definition of the sets $I_{i,k}$.
\item Let $x \in I_{i,k}$ for some $i$.

If $h$ is an integer such that $0 \leq h \leq 3 \left\lfloor \frac{i}{2} \right\rfloor -2$, then we can write
\begin{equation*}
-3 \left\lfloor \frac{i}{2} \right\rfloor + 2 + h = - 3 q + 2 r
\end{equation*}
for some $\{q, r \} \subseteq \N$ such that
\begin{equation*}
0 \leq q \leq \left\lfloor \frac{i}{2} \right\rfloor \quad \text{ and } \quad 0 \leq r \leq 1.
\end{equation*}
Therefore $x = (s+2q) a - 3 q + 2 r$, where $s:=i - 2 q$.

Now we notice that
\begin{align*}
I_{2k,k} & = \{(2k-1)a + 2 \} \cup [(2k-1)a+4, 2ka],\\
I_{2k+1,k} & = \{2ka + 2 \} \cup [2ka+4, (2k+1)a],
\end{align*}
and
\begin{align*}
(2k+1)a + 1 & = (2k-1)a + 2 + c,\\
(2k+1)a + 2 & = (2k-1)a + 5 + b,\\
(2k+1)a + 3 & = (2k-1)a + 4 + c.
\end{align*}
Therefore $\Ap(G,a) \subseteq [a,(2k+1)a+3]$ and consequently $[2ka + 4, \infty[ \subseteq G$.

If we take $x \in G$, namely
\begin{equation*}
x = (s+2q) a - 3 q + 2 r,
\end{equation*}
where $\{q, r, s \} \subseteq \N$ with $0 \leq r \leq q$, then $x \in I_{i,k}$, where $i := s + 2 q$. In fact, if $s = 0$, then $q = \frac{i}{2}$, while $q \leq \left\lfloor \frac{i}{2} \right\rfloor$ if $s > 0$. Moreover, in any case we have that
\begin{equation*}
-3 q + 2 r \in \left\{ - 3 \left\lfloor \frac{i}{2} \right\rfloor \right\} \cup \left[ - 3 \left\lfloor \frac{i}{2} \right\rfloor+2, 0 \right].
\end{equation*}

Therefore $G \subseteq H_{9,k}$.

\item For any $i \in [1,2k]$ we have that
\begin{displaymath}
|I_{i,k}| \equiv 
\begin{cases}
1 \Mod{3} & \text{if $i=1$;}\\
0 \Mod{3} & \text{if $i \not = 1$.}
\end{cases}
\end{displaymath}

The sequence formed by the greatest element of $I_{i,k}$ and the two smallest elements of $I_{i+1,k}$ reads as follows (modulo $3$):
\begin{equation*}
(\overline{i a}, \overline{(i+1)a}, \overline{(i+1)a + 2}) = (\overline{2i}, \overline{2i+2}, \overline{2i+1}).
\end{equation*}
Moreover the sequence of the remaining elements in $\overline{g \cap I_{i,k}}$ is given by a concatenation of $3$-permutations for any $i$. 

Finally we notice that the two smallest elements of $I_{2k+1,k}$ (see above) are $2ka+2$ and $2k a+4$ and the elements in $[2 k a + 5, \infty[$ can be obtained (modulo $3$) through infinitely many concatenations of $3$-permutations. Therefore $H_{9,k}$ is a $3$-permutation semigroup.
\end{enumerate} 
\end{proof}

\begin{lemma}\label{family_10}
Let $k$ be a positive integer, $a:=6k+3$ and $t := k - 1$.

Let $S := \left\{a, a + 2, 2 a - 2 \right\}$ and
\begin{equation*}
H_{10,k} := \{ 0 \} \cup (\cup_{i=0}^{t} (A_{i,k} \cup B_{i,k}) )  \cup (\cup_{i=t+1}^{2k} (C_{i,k} \cup D_{i,k}) ) \cup E,
\end{equation*}
where
\begin{align*}
A_{i,k} & := \left[(1 + 2 i) a - 2 i, (1 +2i) a - 2 i + (6i+2)  \right]_2,\\
B_{i,k} & := \left[(2 + 2 i) a - 2 (i+1), (2 + 2 i) a - 2 (i+1) + (6i+6) \right]_2,
\end{align*}
for any $i \in [0,t]$,
\begin{align*}
C_{i,k} := & [(1+2i) a - 2 i, (1+2i) a - 2 i + 6 (i-t-1) - 2] \\
& \cup  \left[(1 + 2 i) a - 2 i + 6 (i-t-1), (1 + 2 i) a - 2 i + 6 k  \right]_2,\\
D_{i,k} := & [(2+2i) a - 2 (i+1), (2+2i) a - 2 (i+1) + 6 (i-t-1) + 2] \\
& \cup  \left[(2 + 2 i) a - 2 (i+1) + 6 (i-t-1) + 4, (2 + 2 i) a - 2 (i+1) + (6k+2)  \right]_2,\\
\end{align*}
for any $i \in [t+1,2k]$, while
\begin{equation*}
E:=\left[(2+4k) a - 2 (2k+1), \infty \right[.
\end{equation*}

Then the following hold.

\begin{enumerate}
\item $x \in G$ if and only if 
\begin{equation*}
x = (q + 2 u) a + 2 r - 2 u,
\end{equation*}
where $\{q, r, u \} \subseteq \N$ and $0 \leq r \leq q$.
\item $G = H_{10,k}$.
\item The following inequalities hold:
\begin{displaymath}
\begin{array}{ll}
A_{i,k} < B_{i,k} & \text{for any $i \in [0,t]$,}\\
B_{i,k} < A_{i+1,k} & \text{for any $i \in [0,t-1]$,}\\
B_{t,k} < C_{t+1,k},\\
C_{i,k} < D_{i,k} & \text{for any $i \in [t+1,2k]$,}\\
D_{i,k} < C_{i+1,k} & \text{for any $i \in [t+1, 2k-1]$,}\\
D_{2k,k} < E.
\end{array}
\end{displaymath}
\item $H_{10,k}$ is a $3$-permutation semigroup.
\end{enumerate}
\end{lemma}

\begin{proof}
\begin{enumerate}[leftmargin=*]
\item We have that $x \in G$ if and only if $x = s a + t (a+2) + u (2a-2)$ for some  $\{s, t, u \} \subseteq \N$.

According to Lemma \ref{lem_arith_seq} we can write
\begin{align*}
x & = q a + 2 r + u (2a -2) = (q +2 u) a + 2 r - 2 u
\end{align*}
for some $\{q, r \} \subseteq \N$ with $0 \leq r \leq q$.

\item Let $x \in A_{i,k}$ for some $i$.  Then
\begin{equation*}
x = (1+2i) a - 2 i + 2 h
\end{equation*}
for some $h \in [0,3i+1]$.

If $x < (1+2i) a$, then $-2 i +2 h < 0$. In such a case we can write 
\begin{equation*}
-2i+2h = -2 \tilde{i} + 2 \tilde{h},
\end{equation*}
where $\tilde{i} \in [1,i]$ and $\tilde{h} \in [0,1]$. Therefore
\begin{align*}
x & = (1+2i) a - 2 \tilde{i} + 2 \tilde{h} = (1 +2 (i - \tilde{i}) + 2 \tilde{i}) a + 2 \tilde{h} - 2 \tilde{i},
\end{align*}
namely $x \in G$.

If $x \geq (1 +2 i) a$, then $x \in G$, since
\begin{equation*}
x = (q+2u) a + 2 r - 2u,
\end{equation*}
where
\begin{align*}
q & := 1 + 2 i,\\
u & := 0,\\
r & := h - i \leq 3 i + 1 - i = q.
\end{align*}

If $x \in B_{i,k}, C_{i,k}$ or $D_{i,k}$, then we can show that $x \in G$ using an argument as above.

Now we notice that $D_{k,k}$ contains all integers in the interval 
\begin{equation*}
[(2+4k)a - 2 (2k+1), (2+4k)a + 2k].
\end{equation*}
Since in such an interval there are $6 k + 3$ integers, we can say that $\Ap(G,a) \subseteq [a,(2+4k)a + 2k]$. Therefore $E \subseteq G$.

Vice versa, let $x \in G$. Then 
\begin{equation*}
x = (q + 2 u) a + 2 r - 2 u,
\end{equation*}
where $\{q, r, u \} \subseteq \N$ with $0 \leq r \leq q$. 

If $q$ is odd, then $q = 1 + 2 \tilde{q}$ for some $\tilde{q} \in \N$ and
\begin{align*}
x & = (1 + 2 (u + \tilde{q})) a + 2 r - 2 u\\
& = (1 + 2 (u+\tilde{q})) a  - 2 (u + \tilde{q}) + 2 (r + \tilde{q}).
\end{align*}

Before dealing with some different cases, we notice that
\begin{equation*}
r + \tilde{q} \leq q + \tilde{q} = 1 + 3 \tilde{q} \leq 1 + 3 (u+ \tilde{q}).
\end{equation*}

Let $i:= u + \tilde{q}$.
\begin{itemize}
\item If $0 \leq i \leq t$, then $x \in A_{i,k}$.

\item If $t+1 \leq i \leq 2 k$ and $i \leq k$, then $x \in C_{i,k}$.

\item If $t+1 \leq i \leq 2 k$ and $i > k$, then
\begin{align*}
x & \leq (1+2i) a - 2 i + 2 (1+3i) \\
& = (2 + 2i) a - 2 (i+1) - a + 2 (2+3i)
\end{align*}
with
\begin{equation*}
-a + 2 (2 + 3i) \leq 6 (i-t-1) +2.
\end{equation*}
Therefore $x \in D_{i,k}$.

\item If $i \geq 2k+1$, then 
\begin{align*}
x & \geq (1+2i) a - 2 i = 2i (a-1) + a\\
&  \geq (4k+2) (a-1) + a \geq (4k+2) (a-2) \\
& = (2 + 4k) a - 2 (2k+1).
\end{align*}
Therefore $x \in E$.

The proof for $q$ even can be done in a similar way to the odd case.

\end{itemize}

\item All inequalities follow from the definition of the sets.

\item For any $i \in [0,t]$ we have that
\begin{align*}
|A_{i,k}| & = 3 i +2 ,\\
|B_{i,k}| & = 3 i + 4,\\
|A_{i,k} \cup B_{i,k}| & \equiv 0 \Mod{3}.
\end{align*}
The sequence formed by the greatest two elements of $A_{i,k}$ and the smallest element of $B_{i,k}$ reads as follows (modulo $3$):
\begin{equation*}
(\overline{-2 i + 2 (3i)}, \overline{-2 i + 2 (3i+1)}, \overline{-2 (i+1)}) = (\overline{-2 i}, \overline{-2 i +2}, \overline{-2i +1}).
\end{equation*}
This latter is a $3$-permutation and the sequences of the remaining elements of $\overline{g \cap A_{i,k}}$ and $\overline{g \cap B_{i,k}}$ are obtained through concatenations of $3$-permutations.

As regards the sets $C_{i,k}$ and $D_{i,k}$, we have that
\begin{align*}
|C_{t+1,k}| & = 3 k + 1,\\
|D_{t+1,k}| & = 3 k + 3,\\ 
\end{align*}
while 
\begin{align*}
|C_{i,k}| & = 3(i-t-1) + 3 k,\\
|D_{i,k}| & = 3(i-t-1) + 3 k + 3,
\end{align*}
for any $i \in [t+2,2k]$.

If $i \in [t+1,2k]$, then the sequence formed by the greatest element of $C_{i,k}$ and the two smallest elements of $D_{i,k}$ reads as follows (modulo $3$):
\begin{equation*}
(\overline{-2i}, \overline{-2i+1}, \overline{-2i+2}).
\end{equation*} 
 
If $i \in [t+1,2k-1]$, then the sequence formed by the greatest element of $D_{i,k}$ and the two smallest elements of $C_{i+1,k}$ reads as follows (modulo $3$):
\begin{equation*}
(\overline{-2i}, \overline{-2i+1}, \overline{-2i+2}).
\end{equation*}  

Finally, the sequence formed by the greatest element of $D_{2k,k}$ and the two smallest elements of $E$ reads as follows (modulo $3$):
\begin{equation*}
(\overline{2k}, \overline{2k+1}, \overline{2k+2}).
\end{equation*}
Therefore $H_{10,k}$ is a $3$-permutation semigroup.
\end{enumerate}
\end{proof}

\begin{lemma}\label{family_11}
Let  $k$ be a positive integer, $a:=6k+4$ and $t := \frac{a-2}{2}$.

Let $S := \left\{a, b, c \right\}$, where $b:=a + 1$ and $c := a + \frac{a}{2}$.

Let
\begin{equation*}
H_{11,k} := (\cup_{i=0}^{t} (A_{i,k} \cup B_{i,k}) )  \cup C,
\end{equation*}
where $A_{0,k} := \{ 0 \}$, $B_{0,k} := \emptyset$, while
\begin{align*}
A_{i,k} & := \left[i a, i a + i \right],\\
B_{i,k} & := A_{i-1,k} + \left\{ c \right\},
\end{align*}
for any $i \in [1,t]$, and
\begin{equation*}
C := [(t+1)a, \infty[.
\end{equation*}
Then the following hold.
\begin{enumerate}
\item $x \in G$ if and only if 
\begin{equation*}
x = (q+u) a + u \cdot \frac{a}{2} + r,
\end{equation*}
where $\{q, r, u \} \subseteq \N$ with $0 \leq r \leq q$.
\item $G = H_{11,k}$.
\item The following inequalities hold:
\begin{displaymath}
\begin{array}{ll}
A_{i,k} < B_{i,k} & \text{for any $i \in [0,t]$,}\\
B_{i,k} < A_{i+1,k} & \text{for any $i \in [0,t-1]$,}\\
B_{t,k} < C.
\end{array}
\end{displaymath}
\item $H_{11,k}$ is a $3$-permutation semigroup.
\end{enumerate}
\end{lemma}

\begin{proof}
\begin{enumerate}[leftmargin=*]
\item We have that $x \in G$ if and only if 
\begin{equation*}
x = s a + t b + u c
\end{equation*}
for some $\{s, t, u \} \subseteq \N$. According to Lemma \ref{lem_arith_seq} we can write 
\begin{align*}
x & = q a + r + u \left( a +\frac{a}{2} \right) = (q + u) a + u \cdot \frac{a}{2} + r
\end{align*}
for some  $\{q, r, u \} \subseteq \N$ with $r \in [0,q]$.

\item If  $x \in A_{i,k}$ for some $i \in [0,t]$, then $x \in \langle a, b \rangle \subseteq G$ according to Lemma \ref{lem_arith_seq}. 

If $x \in B_{i,k}$ for some $i \in [0,t]$, then $x \in G$ by the definition of the set $B_{i,k}$. 

Now we notice that
\begin{align*}
A_{t,k} & = \left[ta, ta+ \frac{a}{2} - 1 \right],\\
B_{t,k} & = \left[ ta + \frac{a}{2}, (t+1)a -2 \right].
\end{align*}
Therefore 
\begin{equation*}
[(t+1)a, (t+2)a - 1] = (A_{t,k} \cup B_{t,k}) + \left\{ a, b \right\} \subseteq G,
\end{equation*}
namely $\Ap(G,a) \subseteq [a,(t+2)a-1]$ and consequently $C \subseteq G$.

Vice versa, let $x \in G$.

If $x \geq (t+1)a$, then $x \in C$.

If $x < (t+1) a$, then we can write 
\begin{equation*}
x = (q+u) a + u \cdot \frac{a}{2} + r,
\end{equation*}
where $\{q, r, u \} \subseteq \N$ and $0 \leq r \leq q$.

If $u$ is odd, then $u = 2v +1$ for some $v \in \N$. Therefore
\begin{align*}
x & = (q + 2 v + 1) a + (2 v +1) \frac{a}{2} + r = (q + 3v + 1) a + \frac{a}{2} + r.
\end{align*} 
Hence $x \in B_{q+3v,k}$.

If $u$ is even, then we can show as above that $x \in A_{q+3v,k}$, where $v := \frac{u}{2}$.

\item All inequalities follow immediately by the definition of the sets.

\item For any $i \in [0,t]$ we have that
\begin{align*}
|A_{i,k}| & = i +1,\\
|B_{i,k}| & = i.
\end{align*}

Let $i_1, i_2$ and $i_3$ be three consecutive positive integers with $i_{j} \equiv j \Mod{3}$ for any $j \in \{1, 2, 3 \}$. Then 
\begin{equation*}
|\cup_{j=1}^3 (A_{i_j} \cup B_{i_j})| \equiv 0 \Mod{3}.
\end{equation*}

The sequence formed by the greatest two elements of $A_{i_1,k}$ and the smallest element of $B_{i_1,k}$ reads as follows (modulo $3$):
\begin{equation*}
(\overline{i_1 a + i_1 - 1}, \overline{i_1 a + i_1}, \overline{(i_1-1)a}) = (\overline{1}, \overline{2}, \overline{0}).
\end{equation*}
The sequences of the other elements of $\overline{g \cap A_{i_1,k}}$ and $\overline{g \cap B_{i_1,k}}$ are obtained through concatenations of $3$-permutations.

The sequence formed by the greatest two elements of $B_{i_2,k}$ and the smallest element of $A_{i_3,k}$ reads as follows (modulo $3$):
\begin{equation*}
\left(\overline{i_1 a + i_1 - 1 + c}, \overline{i_1 a + i_1 + c}, \overline{i_3 a} \right) = (\overline{1}, \overline{2}, \overline{0}).
\end{equation*}
The sequences of the remaining elements of $g$ belonging  $A_{i_2,k}$, $B_{i_2, k }$, $A_{i_3,k}$ and $B_{i_3,k}$ are obtained (modulo $3$) through concatenations of $3$-permutations.

Finally we notice that $t \equiv 1 \Mod{3}$. As proved above, we have that $\overline{g \cap (A_{t,k} \cup B_{t,k})}$ is given by a concatenation of $3$-permutations.
\end{enumerate}
\end{proof}

\begin{lemma}\label{family_12}
Let $k$ be a positive integer and $a:=12k+4$.

Let $S := \left\{a, b, c \right\}$, where $b := a + \frac{a}{2} - 3$ and $c := a + \frac{a}{2} -1$.

Let
\begin{equation*}
H_{12,k} := (\cup_{i=0}^{3k} (A_{i,k} \cup B_{i,k}) )  \cup (\cup_{j=1}^{3k+1} (C_{j,k} \cup D_{j,k}) ) \cup E,
\end{equation*}
where $A_{0,k} := \{ 0 \}$,  $B_{0,k} := \emptyset$, and
\begin{align*}
A_{i,k} & := \left[ i a - 6 \left\lfloor \frac{i}{3} \right\rfloor, i a \right]_2,\\
B_{i,k} & := \left[ i a + \frac{a}{2} - 3 - 6 \left\lfloor \frac{i-1}{3} \right\rfloor, i a + \frac{a}{2} - 1   \right]_2,
\end{align*}
for any $i \in \left[1,3k \right]$, while
\begin{align*}
C_{1,k} & := \left[ (3k+1) a - 6 k,  (3k+1) a - 4 \right]_2 \cup [(3k+1) a - 2, (3k+1) a],\\
C_{2,k} & := \left[ (3k+2) a - 6 k,  (3k+2) a - 4 \right]_2 \cup [(3k+2) a - 2, (3k+2) a],
\end{align*}
and
\begin{align*}
C_{j,k}  := & \left[(3 k + j) a - (6k+2), (3 k + j) a - 4 - 6 \left\lfloor \frac{j-1}{3} \right\rfloor  \right]_2 \\
& \cup \left[(3k+j) a - 2 - 6 \left\lfloor \frac{j-1}{3} \right\rfloor,  (3 k + j) a  \right]
\end{align*}
for any $j \in [3,3k+1]$, while
\begin{align*}
D_{j,k}  := & \left[(3 k + j) a + \frac{a}{2} + 1 - (6k+2),(3 k + j) a + \frac{a}{2} - 1  - 6 \left\lfloor \frac{j+1}{3} \right\rfloor  \right]_2 \\
& \cup \left[(3k+j) a + \frac{a}{2} + 1 - 6 \left\lfloor \frac{j+1}{3} \right\rfloor,  (3 k + j) a + \frac{a}{2} - 1 \right]
\end{align*}
for any $j \in [1,3k+1]$ and
\begin{equation*}
E := [(6k+1) a + 3, \infty[.
\end{equation*}

Then the following hold.
\begin{enumerate}
\item $x \in G$ if and only if 
\begin{equation*}
x = (s + q) a + q \cdot \frac{a}{2}  - 3 q + 2 r,
\end{equation*}
where $\{s, q, r \} \subseteq \N$ with $0 \leq r \leq q$.
\item $G = H_{12,k}$.
\item The following inequalities hold:
\begin{displaymath}
\begin{array}{ll}
A_{i,k} < B_{i,k} & \text{for any $i \in [1,3k]$,}\\
B_{i,k} < A_{i+1,k} & \text{for any $i \in [1,3k-1]$,}\\
B_{3k,k} < C_{1,k},\\
C_{j,k} < D_{j,k} & \text{for any $j \in [1,3k+1]$,}\\
D_{j,k} < C_{j+1,k} & \text{for any $j \in [1, 3k]$,}\\
D_{3k+1,k} < E.
\end{array}
\end{displaymath}
\item $H_{12,k}$ is a $3$-permutation semigroup.
\end{enumerate}
\end{lemma}

\begin{proof}
\begin{enumerate}[leftmargin=*]
\item We have that $x \in G$ if and only if 
\begin{align*}
x & = s a + t b + u c = s  a + q  b + 2 r\\
& = (s+q) a + q \cdot \frac{a}{2} - 3 q + 2r 
\end{align*}
for some $\{q,r,s \} \subseteq \N$ with $0 \leq q \leq r$ according to Lemma \ref{lem_arith_seq}.

\item Let $x \in H_{12,k}$. 

If $x \in A_{i,k}$ for some $i \in [1,3k]$, then
\begin{equation*}
x = i a - 6 \left\lfloor \frac{i}{3} \right\rfloor + 2 h
\end{equation*}
for some $h \in \left[0, 3 \left\lfloor \frac{i}{3} \right\rfloor \right]$. 

We can write
\begin{equation*}
- 6 \left\lfloor \frac{i}{3} \right\rfloor + 2 h = - 6 l + 2 r
\end{equation*}
for some integers $l$ and $r$ such that $0 \leq l \leq \left\lfloor \frac{i}{3} \right\rfloor$ and $0 \leq r \leq 2$.
In particular we notice that $r =0$ if $l = 0$. 

If we set
\begin{align*}
s & := i - 3l,\\
q & := 2 l,\\
\end{align*}
then we have that 
\begin{equation*}
x = (s+q)a + q \cdot \frac{a}{2} - 3 q + 2 r,
\end{equation*}
namely $x \in G$.

As regards the sets $B_{i,k}$, we notice that
\begin{equation*}
B_{i,k} = A_{i-1,k} + \left\{ b, c \right\}
\end{equation*}
for any $i \in [1,3k]$.

Now we notice that
\begin{align*}
C_{1,k} & = (A_{3k,k} + \{ a \}) \cup \left(\left\{ 3k a - 6 \left\lfloor \frac{3k}{3} \right\rfloor \right\} + \left\{ b \right\} \right),\\
C_{2,k} & = C_{1,k} + \{ a \},\\
C_{3,k} & = (C_{2,k} + \{ a \}) \cup \left( \left\{ 3 k a - 6 \left\lfloor \frac{3k}{3} \right\rfloor \right\} \right) + \left\{2 c \right\}.
\end{align*}

Now we consider $x \in C_{j,k}$ for some $j \in [3,3k+1]$. If
\begin{equation*}
x = (3 k + j) a - (6k+2) + 2 h
\end{equation*}
for some $h \in [0,3k+1]$, then we can write 
\begin{equation*}
x = (s+q)a + q \cdot \frac{a}{2} - 3 q + 2 r,
\end{equation*}
where 
\begin{align*}
s & := 3 k + j - 3l,\\
q & := 2 l,
\end{align*}
for suitable values of $l$ and $r$ such that 
\begin{align*}
- (6k+2) + 2 h = - 6 l + 2 r
\end{align*}
with $0 \leq l \leq k+1$ and $0 \leq r \leq 2$.

Therefore, if $x \in C_{j,k}$ for some $j \equiv 1 \Mod{3}$ with 
\begin{equation*}
x = (3 k + j) a - (6k+2) + 2 h
\end{equation*}
for some $h \in [0,3k+1]$, then $x \in G$.

Now we take $x \in C_{j,k}$ for some $j \equiv 1 \Mod{3}$ with 
\begin{equation*}
x = (3 k + j) a - 1 - 6 \left\lfloor \frac{j-1}{3} \right\rfloor + 2 h
\end{equation*}
for some $h \in [0,j-1]$. Then
\begin{align*}
x & = (3k+j) a - 1 - 6 \left(\frac{j-1}{3} \right) + 2 h = (3k+j) a + 1 - 2 j + 2h.
\end{align*}
If we set
\begin{align*}
s & := 0,\\
q & := 1 + 2 \left( \frac{3k+j-1}{3} \right),\\
r & := h,
\end{align*}
then 
\begin{equation*}
x = (s+q)a + q \cdot \frac{a}{2} - 3 q + 2 r.
\end{equation*}
In particular we notice that $0 \leq r \leq q$. In fact, 
\begin{align*}
r & \leq j - 1  = 3 \left( \frac{j-1}{3} \right) \leq 2 k +  \frac{j-1}{3}\\
& \leq 2 k + 2 \cdot \frac{j-1}{3} + 1 = q.
\end{align*}

Therefore, if $x \in C_{j,k}$ with $j \equiv 1 \Mod{3}$, then $x \in G$. In the case that $j \equiv 2 \Mod{3}$ (resp. $j \equiv 3 \Mod{3}$), then 
\begin{equation*}
C_{j,k} = C_{j-1,k} + \{ a \} \quad \text{(resp. $C_{j-2,k}  + \{ a \}$)}. 
\end{equation*}

As regards the sets $D_{j,k}$, we have that
\begin{align*}
D_{1,k} & = A_{3k,k} + \left\{b, c \right\}
\end{align*}
and
\begin{align*}
D_{j,k} & = C_{j-1,k} +  \left\{b, c \right\}
\end{align*}
for any $j \in [2,3k+1]$.

Now we notice that 
\begin{align*}
C_{3k+1,k} & = \left[ (6k+1) a - \frac{a}{2}, (6k+1)a \right],\\
D_{3k+1,k} & = \{(6k+1)a + 1 \} \cup \left[(6k+1)a + 3, (6k+1) a + \frac{a}{2} -1 \right],
\end{align*}
namely
\begin{equation*}
C_{3k+1,k} \cup D_{3k+1,k} = \left[(6k+1) a - \frac{a}{2}, (6k+1) a + \frac{a}{2} -1 \right] \backslash \{(6k+1)a + 2 \}.
\end{equation*}
Moreover we have that
\begin{equation*}
(6k+2)a + 2 = (6k+1) a + 3 + c.
\end{equation*}
Therefore $\Ap(G,a) \subseteq [a,(6k+2)a+2]$ and $E  \subseteq G$.

Vice versa, let $x \in G$.

If $x \geq (6k+1)a + 3$, then $x \in E \subseteq H_{12,k}$. 

If $x < (6k+1)a +3$, then
\begin{equation*}
x = (s+q) a + q \cdot \frac{a}{2} - 3 q + 2 r,
\end{equation*}
where $\{s, q, r \} \subseteq \N$ with $0 \leq r \leq q$.

We deal in detail with the case that $q$ is even, namely $q = 2 \tilde{q}$ for some $\tilde{q} \in \N$.

We can write
\begin{equation*}
x = (s + 3 \tilde{q}) a - 6 \tilde{q} + 2 r.
\end{equation*} 

We set $i := s + 3 \tilde{q}$ and consider some cases. 

\begin{itemize}
\item \emph{Case 1:} $i \leq 3k+2$. 

We notice that $x \in A_{i,k} \cup C_{1,k} \cup C_{2,k}$ because 
\begin{equation*}
\tilde{q} \leq \left\lfloor \frac{i}{3} \right\rfloor \leq k.
\end{equation*}

\item \emph{Case 2:} $3 k + 3 \leq i \leq 6 k +1$. 

We have that 
\begin{equation*}
i = 3 k + j 
\end{equation*}
for some $j \in [3,3k+1]$. 

If $- 6 \tilde{q} + 2 r \geq - (6k+2)$, then $x \in C_{j,k}$. 

Conversely we have that
\begin{equation*}
- 6 k - 6 \left\lfloor \frac{j}{3} \right\rfloor \leq - 6 \tilde{q} + 2 r \leq - 6 k - 3.
\end{equation*}
Hence
\begin{align*}
x & = i a - a + a - 6 \tilde{q} + 2 r\\
& \geq (i - 1) a + a - 6 k- 6 \left\lfloor \frac{j}{3} \right\rfloor \\
& = (i - 1) a + a - \frac{a}{2} + 2 - 6 \left\lfloor \frac{j}{3} \right\rfloor\\
& = (3k+j-1) a + \frac{a}{2} + 2 - 6 \left\lfloor \frac{j}{3} \right\rfloor,  
\end{align*}
namely $x \in D_{j-1,k}$.

\item \emph{Case 3:} $i \geq  6k+2$.

First we notice that
\begin{align*}
x & \geq i a - 6 \tilde{q} \geq i a - 6 \left\lfloor \frac{i}{3} \right\rfloor.
\end{align*}

Moreover the sequence 
\begin{equation*}
\left\{ \delta_i \right\}_{i=6k+2}^{\infty} := \left\{i a - 6 \left\lfloor \frac{i}{3} \right\rfloor \right\}_{i=6k+2}^{\infty}
\end{equation*}
is increasing. Indeed, if $i \in [6k+2, \infty[$, we have that
\begin{align*}
\delta_{i+1} - \delta_i = a - 6 \left( \left\lfloor \frac{i+1}{3} \right\rfloor - \left\lfloor \frac{i}{3} \right\rfloor \right) \geq a - 6 > 0.
\end{align*}

Hence we conclude that 
\begin{align*}
x & \geq i a - 6 \left\lfloor \frac{i}{3} \right\rfloor \geq (6k+2) a - 6 \left\lfloor \frac{6k+2}{3} \right\rfloor \\
& = (6k+2) a - 12 k = (6k+1) a + 4,
\end{align*}
namely $x \in E$.
\end{itemize}

\item All inequalities follow from the definition of the sets.

\item For any $i \in [1,3k]$ we have that
\begin{align*}
|A_{i,k}| & = 3 \left\lfloor \frac{i}{3} \right\rfloor + 1,\\
|B_{i,k}| & =  \left\lfloor \frac{i-1}{3} \right\rfloor + 2.
\end{align*}
\end{enumerate}
Therefore $|A_{i,k} \cup B_{i,k}| \equiv 0 \Mod{3}$. The sequence formed by the greatest element of $A_{i,k}$ and the smallest two elements of $B_{i,k}$ is (modulo $3$) a $3$-permutation and the sequences of the remaining elements of $\overline{g \cap A_{i,k}}$ and $\overline{g \cap B_{i,k}}$ can be obtained through concatenations of $3$-permutations. Since the same holds for the elements in $C_{j,k}$ and $D_{j,k}$ for any $j$, we conclude that $H_{12,k}$ is a $3$-permutation semigroup. 
\end{proof}

\begin{lemma}\label{family_13}
Let $k$ be a positive integer, $a:=6k+5$, $b:=a + \frac{a-3}{2}$ and $c:=b+1$. 

Let $S := \{a, b, c \}$ and
\begin{equation*}
H_{13,k} := \{ 0 \} \cup \left( \cup_{i=0}^k I_{i,k} \right) \cup C,
\end{equation*}
where $I_{i,k} := A_{i,k} \cup B_{i,k}$ and 
\begin{align*}
A_{i,k} & := \left( [a, a + 3 i \right] \cup [b, b + 1+3i] \cup [2a, 2a+3i]) + \{ 2 i b \},\\
B_{i,k} & := \left( [a, a+1+3i] \cup [b, b+3+3i] \cup [2a, 2a+1+3i] \right) + \{(2 i + 1 ) b \},
\end{align*}
for any $i \in [0,k]$, while
\begin{equation*}
C := [2 a + 2 k b, \infty[.
\end{equation*}
Then the following hold.
\begin{enumerate}
\item $H_{13,k}$ is a co-finite submonoid of $G$ containing $S$.
\item $H_{13,k}$ is a $3$-permutation semigroup.
\end{enumerate}
\end{lemma}

\begin{proof}
We notice that the claim is true when $k=1$ by a direct computation (see Table 1 in Section \ref{sec_n_pns}). 

In the following we suppose that $k \geq 2$.

\begin{enumerate}[leftmargin=*]
\item We prove by induction on $i$ that 
\begin{equation*}
A_{i,k} \cup B_{i,k} \subseteq G.
\end{equation*}

First we observe that 
\begin{equation*}
A_{0,k} = \{ a \} \cup \{b, c \} \cup \{ 2a \} = S \cup \{ 2 a \} \subseteq G.
\end{equation*}

Since
\begin{align*}
[a+b,a+b+1] & = \{ a \} + \{b, c \},\\
[2b,2b+3] & = ([b,c] + [b,c]) \cup \{ 3a \},\\
[2a+b,2a+b+1] & = [a+b,a+b+1] + \{ a  \},
\end{align*}
we can say that $B_{0,k} \subseteq G$.

Now let $i > 0$. We have that
\begin{align*}
[a,a+3i] + \{ 2 i b \} = & ([2b, 2b+3+3(i-1)] + \{2 (i-1) b \} ) + \{ a  \},\\
[b,b+1+3i] + \{ 2 i b \} = & (([2b, 2b+3+3 (i-1)] + \{ 2 (i-1) b \}) + \{ b \})\\
& \cup ((\{ a + b +1 + 3 (i-1) \} + \{ 2 (i-1) b \}) + \{ 2 a \},\\
[2a,2a+3i] + \{ 2 i b \} = & ([a,a+3i] + \{ 2 i b \}) + \{ a \}. 
\end{align*}
Therefore $A_{i,k} \subseteq B_{i-1,k} + \{a, b, 2a \}  \subseteq G$.

In a similar way we can prove that $B_{i,k} \subseteq A_{i-1,k} + \{ a, b, c \}$. 

As regards the set $C$, we notice that
\begin{align*}
2 a + 3 k & = a + b -1,\\
a + b + 1 + 3 k & = 2 b,\\
2 b + 3 & =  3 a.
\end{align*}
Therefore $[2a+2kb,3a+2kb] \subseteq G$, namely $\Ap(G,a) \subseteq [a,3a+2kb]$. We conclude that $C \subseteq G$.

Now we take $x \in I_{i_1,k}$ and $y \in I_{i_2,k}$ for some $\{i_, i_2 \} \subseteq [0,k]$. If $x + y < 2 a + 2 k b$, then $x+y$ belongs to one of the sets in rows $2-3$, columns $2-3$ of the table below, where $i_3 := i_1+ i_2$.
\begin{displaymath}
\begin{array}{c|c|c}
& A_{i_2,k} & B_{i_2,k} \\
\hline
A_{i_1,k} & I_{i_3,k} & I_{i_3,k} \cup I_{i_3 + 1, k}\\
\hline
B_{i_1,k} & & I_{i_3+1,k} \cup I_{i_3 + 2, k} 
\end{array}
\end{displaymath}

\item The following inequalities hold for any $i \in [0,k-1]$:
\begin{itemize}
\item $a +3 i < b$;
\item $b + 1+ 3 i < 2a$;
\item $2 a + 3 i < a + b$;
\item $a + b + 1 + 3 i < 2 b$;
\item $2 b + 3 +3 i < 2 a + b$;
\item $2 a + b + 1 +3 i +2 i b < a + 2 (i+1) b $.
\end{itemize}

Moreover we have that
\begin{itemize}
\item $a + 3 k< b$;
\item $b + 1 + 3 k< 2a$. 
\end{itemize}
\end{enumerate}

We notice that
\begin{equation*}
A_{i,k} < B_{i,k} \quad \text{ and } B_{i,k} < A_{i+1,k}
\end{equation*}
and
\begin{equation*}
|A_{i,k}| \equiv 1 \Mod{3}, \quad B_{i,k} \equiv 2 \Mod{3}
\end{equation*}
for any $i \in [0,k-1]$. Therefore 
\begin{equation*}
|A_{i,k} \cup B_{i,k}| \equiv 0 \Mod{3}
\end{equation*}
for any $i \in [0,k-1]$. From the definition of the sets $A_{i,k}$ and $B_{i,k}$ one can verify that the sequence $\overline{g \cap (A_{i,k} \cup B_{i,k})}$ is given by a concatenation of $3$-permutations for any $i$. The same holds for the sequence $\overline{g \cap (A_{k,k} \cup C)}$. Hence we conclude that $H_{13,k}$ is a $3$-permutation semigroup.
\end{proof}

\begin{lemma}\label{family_14}
Let $k$ be a positive integer, $a:=6k+5$, $b:=2a-6$ and $c:=2a-4$.

Let $S:=\{a, b, c \}$ and
\begin{equation*}
H_{14,k} := \left( \cup_{i=0}^{4k+2} I_{i,k} \right) \cup [(4k+1)a + 3, \infty[,
\end{equation*}
where 
\begin{align*}
I_{i,k} & := \left[ ia - 6 \left\lfloor \frac{i}{2} \right\rfloor, i a - 4 \right]_2 \cup \{ia  \}
\end{align*}  
for any $i \in [0,4k+2]$.

Then the following hold.
\begin{enumerate}
\item $x \in G$ if and only if 
\begin{equation*}
x = (s+2q) a - 6 q + 2 r,
\end{equation*}
where $\{q, r, s \} \subseteq \N$ with $0 \leq r \leq q$.
\item We have that
\begin{displaymath}
\begin{array}{ll}
I_{i,k} < I_{i+2,k} & \text{if $i \in [0,4k]$,}\\
I_{i,k} < I_{i+1,k} & \text{if $i \in [0, 2k]$,}\\
\end{array}
\end{displaymath}
while for any $i \in [2k+1,4k+1]$ we have that
\begin{equation*}
I_{i+1,k} \cap [(i-1)a + 1, ia] = \left[ia - 6 i_k - 1, i a - 1 \right]_2,
\end{equation*}
where $i_k := \left\lfloor \frac{i-1-2k}{2} \right\rfloor$.
\item $G = H_{14,k}$.
\item $H_{14,k}$ is a $3$-permutation semigroup.
\end{enumerate}
\end{lemma}
\begin{proof}
We notice that the claim is true when $k=1$ by a direct computation (see Table 1 in Section \ref{sec_n_pns}). 

In the following we suppose that $k \geq 2$.

\begin{enumerate}[leftmargin=*]
\item The assertion follows from Lemma \ref{lem_arith_seq}.
\item All assertions follow from the definition of the sets.
\item Let $x \in I_{i,k}$ for some $i \in [0,4k+2]$. 

If $x = ia$, then $x \in G =\langle S \rangle$.

If $x < ia$, then
\begin{equation*}
x = i a - 6 \left\lfloor \frac{i}{2} \right\rfloor + 2 h
\end{equation*}
for some integer $h$ such that $0 \leq h \leq 3 \left\lfloor \frac{i}{2} \right\rfloor - 2$.

We can write
\begin{equation*}
- 6 \left\lfloor \frac{i}{2} \right\rfloor + 2 h = - 6 q + 2 r
\end{equation*}
for some $\{q, r \} \subseteq \N$ such that 
\begin{equation*}
0 \leq q \leq \left\lfloor \frac{i}{2} \right\rfloor \quad \text{ and } \quad  0 \leq r \leq 2.
\end{equation*}
In particular we notice that $r = 0$ if $q = 0$, while $r \leq 1$ if $q =1$ due to the restrictions on $h$. Therefore $r \leq q$ whichever the value of $q$ is.

Hence $x = (s+2q) a - 6 q + 2 r$, where $s := i - 2 q$.

Now we want to prove that any $x \geq (4k+1)a+3$ belongs to $G$. 

First we notice that
\begin{align*}
I_{4k+1,k} \cap [4k a + 1, (4k+1)a] & = [4ka + 1, 4 ka + 1 + 6k]_2,\\
I_{4k+2,k} \cap [4k a + 1, (4k+1)a] & = [4ka + 4, 4 k a + 4 + 6 k]_2.
\end{align*}
Therefore 
\begin{align*}
\{4ka+1\} \cup [4ka+3,(4k+1)a-3] \cup \{(4k+1)a - 1 \}  & \subseteq G.
\end{align*}
Since
\begin{equation*}
[4ka+4,4ka+8] + \{ b \} = [(4k+2)a -2, (4k+2)a+2],
\end{equation*}
we can say that 
\begin{equation*}
[(4k+1)a+3,(4k+2)a+2] \subseteq G.
\end{equation*}
Hence $\Ap(G,a) \subseteq [a,(4k+2)a+2]$ and $[(4k+1)a+3, \infty[ \subseteq G$.

Now we take $x \in G$. We have that
\begin{equation*}
x = (s+ 2 q) a - 6 q + 2 r
\end{equation*}
for some $\{q, r, s \} \subseteq \N$ with $0 \leq r \leq q$.

Let $i := s + 2q$. Then $x \in I_{i,k}$. In fact, $q = \frac{i}{2}$ if $s = 0$, while $q \leq \lfloor \frac{i}{2} \rfloor$ if $s > 0$, and 
\begin{equation*}
- 6 q + 2 r = 0 \quad \text{ or } \quad -6q +2 r \leq - 4.
\end{equation*}

\item For any $i \in [1,4k+2]$ we define $U_i := [(i-1)a +1, i a] \cap G$.

If $i \in [1,2k]$, then 
\begin{displaymath}
|U_i| \equiv
\begin{cases}
1 \Mod{3} & \text{if $i=1$,}\\
0 \Mod{3} & \text{if $i \not = 1$.}
\end{cases}
\end{displaymath}
For any $i \in [2,2k]$ the sequence (modulo $3$) formed by the greatest element of $U_{i-1}$ and the two smallest elements of $U_{i}$ reads as follows:
\begin{equation*}
(\overline{2i-2}, \overline{2i}, \overline{2i+2}).
\end{equation*} 
The sequence of the remaining elements of $\overline{g \cap U_{i}}$ can be obtained through a concatenation of $3$-permutations.

If $i \in [2k+1,4k+1]$, then 
\begin{align*}
U_i & = (I_{i,k} \cap [(i-1)a+1, ia-6 i_k -4]) \cup [ia-6 i_k -2, i a -3] \cup \{ ia - 1, i a \}.
\end{align*}
In particular, if $i \in \{ 2k+1, 2k+2 \}$, then
\begin{align*}
U_i & = (I_{i,k} \cap [(i-1)a+1, ia - 4]) \cup \{ ia - 1, i a \}.
\end{align*}

For $i = 2k+1$ we have that
\begin{equation*}
|(I_{i,k} \cap [(i-1)a+1, ia - 4])| \equiv 2 \Mod{3}.
\end{equation*}

For any $i \in [2k+2,4k+1]$ we have that
\begin{equation*}
|I_{i,k} \cap [(i-1)a+1, ia-6 i_k -4]| \equiv 1 \mod{3},
\end{equation*}
while 
\begin{equation*}
|[ia-6 i_k -2, i a -3]| \equiv 0 \mod{3},
\end{equation*}
provided that $i > 2k+2$. Therefore
\begin{equation*}
|U_i| \equiv 0 \Mod{3}
\end{equation*}
for any $i \in [2k+3,4k+1]$.

Now we  observe that the sequence (modulo $3$) formed by the greatest element of $U_{2k}$ and the two smallest elements of $U_{2k+1}$ reads as follows:
\begin{equation*}
(\overline{k}, \overline{k+2}, \overline{k+1}).
\end{equation*}  

For any $i \in [2k+2,4k+2]$ we have that the sequence (modulo $3$) formed by the greatest two elements of $U_{i-1}$ and the smallest element of $U_{i}$ reads as follows:
\begin{equation*}
(\overline{(i-1)a-1}, \overline{(i-1)a}, \overline{(i-1)a+1}).
\end{equation*}

The sequence of the remaining elements of $\overline{g \cap U_i}$ is given by a concatenation of $3$-permutations for any $i \in [2k+1,4k+1]$. 

As regards the set $U_{4k+2}$, we notice that
\begin{equation*}
U_{4k+2} =\{ (4k+1)a+1 \} \cup [(4k+1)a+3, (4k+2)a].
\end{equation*}
Therefore the sequence of elements greater than $(4k+1)a+1$ belonging to $H_{14,k}$ reads (modulo $3$) as an infinite concatenation of $3$-permutations.
\end{enumerate}
\end{proof}

\begin{lemma}\label{family_15}
Let $k$ be a positive integer, $a:=12k+8$, $b:= a + 2$ and $c := b + \frac{b}{2}$.

Let $S := \left\{a, b, c \right\}$ and
\begin{equation*}
H_{15,k} := (\cup_{i=0}^{6k+4} (A_{i,k} \cup B_{i,k}) ) \cup C, 
\end{equation*}
where $A_{0,k} := \{ 0 \}$,  $B_{0,k} := \emptyset$, and
\begin{align*}
A_{i,k} & := [ i a, ia + 2 i]_2,\\
B_{i,k} & := \left[i a + \frac{a}{2} + 3, ia + \frac{a}{2} + 3 + 2 (i-1) \right]_2,
\end{align*}
for any $i \in [1,6k+4]$, while
\begin{equation*}
C := \left[(6k+4) a + \frac{a}{2} + 3, \infty \right[.
\end{equation*}
If 
\begin{align*}
U_{i,k} & := \left[ i a, ia + \frac{a}{2} + 2 \right],\\
V_{i,k} & := \left[i a + \frac{a}{2} + 3, (i+1) a - 1 \right],
\end{align*}
for any $i \in \left[1, 6k + 4 \right]$, then the following hold.

\begin{enumerate}
\item $x \in G$ if and only if 
\begin{equation*}
x = (q+ 3v + \epsilon) a + \epsilon \cdot \frac{a}{2}  + 6 v + 3 \epsilon + 2 r,
\end{equation*}
where $\{q, r, v \} \subseteq \N$, $\epsilon \in \{0, 1 \}$ and $0 \leq r \leq q$.
\item $G = H_{15,k}$.
\item The following inequalities hold:
\begin{displaymath}
\begin{array}{ll}
U_{i,k} < V_{i,k} & \text{for any $i$,}\\
A_{i,k} < A_{i+1,k} & \text{if $i \in [1,6k+3]$,}\\
B_{i,k} < B_{i+1,k} & \text{if $i \in [1,6k+3]$,}\\
V_{i,k} < U_{i+1,k} & \text{if $i \in [1,6k+3]$,}\\
A_{i,k} < V_{i,k} & \text{if $i \in [1,3k+3]$,}\\
A_{i,k} \cap V_{i,k} = \left[ia + \frac{a}{2} + 3, ia + 2 i \right]_2 & \text{if $i \in [3k+4,6k+4]$},\\
B_{i,k} < U_{i+1,k} & \text{if $i \in [1,3k+1]$,}\\
B_{i,k} \cap U_{i+1,k} = \left[(i+1) a,  ia + \frac{a}{2} + 3 + 2 (i-1) \right]_2 & \text{if $i \in [3k+2,6k+3]$.}
\end{array}
\end{displaymath}
\item $H_{15,k}$ is a $3$-permutation semigroup.
\end{enumerate}
\end{lemma}
\begin{proof}
\begin{enumerate}[leftmargin=*]
\item We have that $x \in G$ if and only if 
\begin{equation*}
x = s a + t (a+2) + u \left( a + 2 + \frac{a+2}{2} \right)
\end{equation*}
for some $\{s, t, u \} \subseteq \N$. According to Lemma \ref{lem_arith_seq} we can write 
\begin{equation*}
x = q a + 2 r + u \left( a + 2 + \frac{a+2}{2} \right)
\end{equation*}
for some $\{q, r, u \} \subseteq \N$ such that $0 \leq r \leq q$. If we write 
\begin{equation*}
u = 2 v + \epsilon
\end{equation*}
for some $v \in \N$ and $\epsilon \in \{0, 1 \}$, then the result follows.

\item If $x \in A_{i,k}$ for some $i \geq 1$, then $x = i a + 2 h$ for some $h \in [0,i]$. Therefore $x = q a + 2 r$, where $q := i$ and $r:=h$, namely $x \in G$.

If $x \in B_{i,k}$ for some $i \geq 1$, then $x \in A_{i-1} + \{ c \} \subseteq G$.

Now we notice that
\begin{align*}
A_{6k+3,k} & = [(6k+3)a, (6k+4)a - 2]_2,\\
B_{6k+3,k} & = \left[ (6k+3)a + \frac{a}{2} + 3, (6k+4) a + \frac{a}{2} -1 \right]_2,\\
A_{6k+4,k} & = [(6k+4) a, (6k+5)a]_2.
\end{align*}
Therefore 
\begin{equation*}
\left[(6k+3) a + \frac{a}{2} + 3, (6k+4) a + \frac{a}{2} \right] \cup \{ (6k+4) a + \frac{a}{2} + 2 \} \subseteq G
\end{equation*}
and consequently
\begin{equation*}
\left[(6k+4) a + \frac{a}{2} + 3, (6k+5) a + \frac{a}{2} \right] \cup \{ (6k+5) a + \frac{a}{2} + 2 \} \subseteq G.
\end{equation*}
Since
\begin{equation*}
(6k+5) a + \frac{a}{2} + 1 = (6k+4) a - 1 + c,
\end{equation*}
we conclude that 
\begin{equation*}
\Ap(G,a) \subseteq \left[a, (6k+5) a + \frac{a}{2}+2 \right].
\end{equation*}
Hence $C \subseteq G$.

Now let $x \in G$.

If $x \geq (6k+4)a + \frac{a}{2} + 3$, then $x \in C$.

Conversely we have that
\begin{equation*}
x = (q+ 3v + \epsilon) a + \epsilon \cdot \frac{a}{2}  + 6 v + 3 \epsilon + 2 r
\end{equation*}
for some $\{q, r, v \} \subseteq \N$, $\epsilon \in \{0, 1 \}$ and $0 \leq r \leq q$. 

If $\epsilon = 0$, then $x = ia+2h$, where
\begin{align*}
i & := q+3v,\\
h & := 3 v + r,
\end{align*} 
with $0 \leq h \leq 3 v + q  = i \leq 6k+4$. Hence  $x \in A_{i,k}$.

If $\epsilon = 1$, then $x = i a + \frac{a}{2} + 3 + 2 h$, where  
\begin{align*}
i & := q+3v+1,\\
h & := 3 v + r,
\end{align*} 
with $0 \leq h \leq 3 v + q  = i - 1 \leq 6 k + 3$. Hence  $x \in B_{i,k}$.

\item All inequalities follow from the definition of the sets.

\item First we notice that 
\begin{align*}
|U_{i,k} \cap G| & = i +1,\\
|V_{i,k} \cap G| & = i,\\
|(U_{i,k} \cup V_{i,k}) \cap G| & = 2 i +1,
\end{align*}
for any $i \in [1,3k+1]$.

Now we take a set of consecutive integers $\{i_1, i_2, i_3 \}$ such that $1 \leq i_1 < i_2 < i_3 \leq 3 k + 1$ and $i_j \equiv j \Mod{3}$ for any $j \in \{1, 2, 3 \}$. 
For any such an index $j$ we have that
\begin{align*}
A_{i_j,k} & = U_{i_j,k} \cap G,\\
B_{i_j,k} & = V_{i_j,k} \cap G.
\end{align*} 

If $j =1$, then $|(U_{i_1,k} \cup V_{i_1,k}) \cap G|  \equiv 0 \Mod{3}$ and the sequence formed by the greatest two elements of $A_{i_1,k}$ and the smallest element of $B_{i_1,k}$ is a $3$-permutation according to the definition of such sets. The remaining elements of $\overline{g \cap A_{i_1,k}}$ and $\overline{g \cap B_{i_1,k}}$ can be obtained through concatenations of $3$-permutations. We notice that this argument applies in particular when $i = 3k+1$.

We can also check that the sequence formed by the greatest two elements of $B_{i_2,k}$ and the smallest element of $A_{i_3,k}$ is a $3$-permutation. The sequences of the remaining elements of $g$ belonging to $A_{i_2,k}$, $B_{i_2,k}$, $A_{i_3,k}$ and $B_{i_3,k}$ are obtained (modulo $3$) through concatenations of $3$-permutations too.

Now we notice that 
\begin{align*}
|U_{3k+2,k} \cap G| & = 3 k + 3 \equiv 0 \Mod{3},\\
|V_{3k+2,k} \cap G| & = \frac{a}{4} - 1 \equiv 1 \Mod{3}. 
\end{align*}

As regards the indices $i \in [3k+3,6k+3]$, we have that
\begin{align*}
|U_{i,k} \cap G| & = \left(\frac{a}{4} + 2 \right) + \left(i - \frac{a}{4} \right) = 2 + i,\\
|V_{i,k} \cap G| & = \left(\frac{a}{4} - 1  \right) + \left(i - \frac{a}{4} + 1 \right) = i -2.
\end{align*}

In particular we notice that 
\begin{equation*}
|(U_{i,k} \cup V_{i,k}) \cap G| \equiv 2 i \Mod{3}
\end{equation*}
for any $i \in [3k+2,6k+3]$.

We have that 
\begin{align*}
U_{3k+2,k} \cap G & = \left[ia, ia + \frac{a}{2} \right]_2,\\
V_{3k+2,k} \cap G & = \left[ia + \frac{a}{2} + 3, (i+1)a - 1 \right]_2.
\end{align*}

The sequence formed by the elements of 
\begin{equation*}
((U_{3k+2,k} \cup V_{3k+2,k}) \cap G) \backslash \{(3k+3)a-1 \}
\end{equation*}
is given (modulo $3$) by a concatenation of $3$-permutations. Moreover the sequence formed by $(3k+3)a-1$ and the two smallest elements of $U_{3k+3,k} \cap G$ is (modulo $3$) a $3$-permutation too, since it reads as follows (modulo $3$):
\begin{equation*}
(\overline{(3k+3)a-1}, \overline{(3k+3) a}, \overline{(3k+3)a+1}).
\end{equation*}

Now let $\{i_2, i_0, i_1 \}$ be a set of three consecutive integers such that $3k+2 \leq i_2 < i_0 < i_1 \leq 6 k + 1$ and $i_j \equiv j \Mod{3}$ for any $j$.
Then

\begin{equation*}
\sum_{j=0}^2 |((U_{i_j} \cap G) \cup (V_{i_j} \cap G))| \equiv 0 \Mod{3}.
\end{equation*}

Moreover the sequence of elements in 
\begin{equation*}
\cup_{j=0}^2 ((U_{i_j} \cap G) \cup (V_{i_j} \cap G)) 
\end{equation*}
is given (modulo $3$) by concatenations of $3$-permutations. We analyse in detail just the case $j=2$. 
Since we have already considered the case $i_j = 3k+2$, we can suppose that $i_j > 3k+2$.

The elements of $U_{i_{2},k} \cap G$ are given by the following union of intervals:
\begin{equation*}
\left[i_{2} a, (i_{2} - 1) a + \frac{a}{2} + 2 i_{2} - 1 \right] \cup \left[(i_{2}-1) a + \frac{a}{2} + 2 i_{2}, i_{2} a + \frac{a}{2} + 2 \right]_2.
\end{equation*}

We notice that the sequence (modulo $3$) formed by the elements in
\begin{equation*}
\left[i_{2} a, (i_{2} - 1) a + \frac{a}{2} + 2 i_{2} - 1 \right]
\end{equation*}
is given by a concatenation of $3$-permutations. Moreover 
\begin{equation*}
\left| \left[(i_{2}-1) a + \frac{a}{2} + 2 i_{2}, i_{2} a + \frac{a}{2} + 2 \right]_2 \right| \equiv 1 \Mod{3}.
\end{equation*}
The sequence formed by the greatest element of $U_{i_{2},k} \cap G$ and the two smallest elements in $V_{i_{2},k} \cap G$ reads as follows (modulo $3$):
\begin{equation*}
\left( \overline{i_{2} a + \frac{a}{2} + 2}, \overline{i_{2} a + \frac{a}{2} + 3}, \overline{i_{2} a + \frac{a}{2} + 4} \right) = (\overline{1}, \overline{2}, \overline{0}).
\end{equation*}
In a similar way we can check that the sequence (modulo $3$) of elements in 
\begin{equation*}
(V_{i_2,k} \cap G) \backslash \{(i_2+1) a - 1 \}
\end{equation*}
is given by a concatenation of $3$-permutations. 

Finally we notice that if 
\begin{equation*}
(i_2, i_0,i_1) = (6k+2, 6k+3, 6k+4),
\end{equation*}
then 
\begin{equation*}
|\cup_{j \in \{2,0 \}} ((U_{i_j,k} \cup V_{i_j,k}) \cap G)| \equiv 1 \Mod{3}.
\end{equation*}
The sequence formed by the greatest element of $V_{6k+3,k} \cap G$ and the smallest two elements of $U_{6k+4,k}$ reads (modulo $3$) as a $3$-permutation. The sequence of the elements in $G$ greater than $(6k+4)a+1$ is given by the elements in 
\begin{equation*}
\left[(6k+4) a + 2, (6k+4) a + \frac{a}{2} - 3 \right] \cup \left[(6k+4) a + \frac{a}{2} - 1, \infty \right[.
\end{equation*} 
The elements belonging to this latter union of intervals are given (modulo $3$) by concatenations of $3$-permutations.
\end{enumerate}
\end{proof}

\begin{lemma}\label{family_16}
Let $k$ be a positive integer, $a:=12k+8$, $b:=a + \frac{a}{2} + 1$ and $c := a + \frac{a}{2} + 3$.

Let $S := \left\{a, b, c \right\}$ and
\begin{equation*}
H_{16,k} := (\cup_{i=0}^{3k+2} (A_{i,k} \cup B_{i,k}) ) \cup (\cup_{j=0}^{3k+1} (C_{j,k} \cup D_{j,k})) \cup E, 
\end{equation*}
where $A_{0,k} := \{ 0 \}$,  $B_{0,k} := \emptyset$, and
\begin{align*}
A_{i,k} & := \left[ia,  i a + 6 \left\lfloor \frac{i}{3} \right\rfloor  \right]_2,\\
B_{i,k} & := A_{i-1,k} + \left\{ b, c \right\},
\end{align*}
for any $i \in \left[1, 3k + 2 \right]$, while
\begin{align*}
C_{j,k}  := & \left[(3 k + 3 +  j) a, (3k+3+j)a + 6 \left\lfloor \frac{j+1}{3} \right\rfloor \right] \\
& \cup \left[(3k+3+j) a + 6 \left\lfloor \frac{j+1}{3} \right\rfloor + 2, (3k+3+j) a + \frac{a}{2}   \right]_2,\\
D_{j,k} := & \left[(3 k + 3 +  j) a + \frac{a}{2} + 1, (3k+3+j)a + \frac{a}{2} + 3 + 6 \left\lfloor \frac{j}{3} \right\rfloor \right] \\
& \cup \left[(3k+3+j) a + \frac{a}{2} + 5 + 6 \left\lfloor \frac{j}{3} \right\rfloor, (3k+4+j) a - 1  \right]_2,\\
\end{align*}
for any $j \in [0,3k+1]$ and
\begin{equation*}
E := \left[(6k+4)a  + \frac{a}{2} + 1, \infty \right[.
\end{equation*}

Then the following hold.
\begin{enumerate}
\item $x \in G$ if and only if 
\begin{equation*}
x = (s + q) a + q \cdot \frac{a}{2}  + q + 2 r,
\end{equation*}
where $\{q, r, s \} \subseteq \N$ with $0 \leq r \leq q$.
\item $S \subseteq H_{16,k} \subseteq G$.
\item The following inequalities hold:
\begin{displaymath}
\begin{array}{ll}
A_{i,k} < B_{i,k} & \text{for any $i \in [1,3k+2]$,}\\
B_{i,k} < A_{i+1,k} & \text{for any $i \in [1,3k+1]$,}\\
B_{3k+2,k} < C_{0,k},\\
C_{j,k} < D_{j,k} & \text{for any $j \in [0,3k+1]$,}\\
D_{j,k} < C_{j+1,k} & \text{for any $j \in [0, 3k]$,}\\
D_{3k+1,k} < E.
\end{array}
\end{displaymath}
\item $H_{16,k}$ is a numerical semigroup.
\item $H_{16,k}$ is a $3$-permutation semigroup.
\end{enumerate}
\end{lemma}

\begin{proof}
\begin{enumerate}[leftmargin=*]
\item We have that $x \in G$ if and only if 
\begin{equation*}
x = sa + t \left(a + \frac{a}{2} + 1 \right) + u \left( a + \frac{a}{2} + 3 \right)
\end{equation*}
for some $\{s, t, u \} \subseteq \N$. According to Lemma \ref{lem_arith_seq} we can write
\begin{align*}
x & = sa + q \left(a + \frac{a}{2} + 1 \right) + 2 r = (s+q)a + q \cdot \frac{a}{2} + q + 2 r 
\end{align*}
for some $\{q, r \} \subseteq \N$ with $0 \leq r \leq q$.

\item First we notice that $S \subseteq (A_{1,k} \cup B_{1,k})$.

Let $x \in A_{i,k}$ for some $i \in [1,3k+2]$. Then
\begin{equation*}
x = ia +2 h
\end{equation*}
with $0 \leq h \leq 3 \left\lfloor \frac{i}{3} \right\rfloor$. We can write
\begin{equation*}
2 h = 2 \tilde{h} + 2 r
\end{equation*}
for some $\{ \tilde{h}, r \} \subseteq \N$ such that 
\begin{equation*}
0 \leq \tilde{h} \leq \left\lfloor \frac{i}{3} \right\rfloor \quad \text{and} \quad 0 \leq r \leq 2 \tilde{h}.
\end{equation*}
If we set 
\begin{align*}
s & := i - 3 \tilde{h},\\
q & := 2 \tilde{h}, 
\end{align*}
then
\begin{equation*}
x = (s + q) a + q \cdot \frac{a}{2}  + q + 2 r \in G.
\end{equation*}

If $x \in B_{i,k}$ for some $i \in [1,3k+2]$, then $x \in G$ by definition of the sets $B_{i,k}$.

Now we consider the sets $C_{j,k}$.

First we notice that 
\begin{align*}
C_{0,k} & \subseteq (A_{3k,k} + \{ 3a, 2b, 2 c \}),\\
C_{1,k} & = C_{0,k} + \{ a \},\\
D_{0,k} & \subseteq (B_{3k+2,k} + \{ a \}) \cup (B_{3k+1,k} + \{ c \}),\\
D_{1,k} & = D_{0,k} + \{ a \}. 
\end{align*} 

Then we have that
\begin{align*}
C_{2,k} & \subseteq (C_{1,k} + \{ a \}) \cup (C_{0,k} + \{b, c \}) \cup (D_{0,k} + \{ c \}), \\
D_{2,k} & = D_{1,k} + \{ a \}.
\end{align*}

Finally, for any $j \in [3,3k+1]$ the following hold.
\begin{itemize}
\item If $j \equiv 0 \Mod{3}$, then 
\begin{align*}
C_{j,k} & = C_{j-1,k} + \{ a \},\\
D_{j,k} & \subseteq D_{j-3,k} + \{3a, 2b, 2c \}.
\end{align*}
\item If $j \equiv 1 \Mod{3}$, then 
\begin{align*}
C_{j,k} & = C_{j-1,k} + \{ a \},\\
D_{j,k} & = D_{j-1,k} + \{ a \}.
\end{align*}
\item If $j \equiv 2 \Mod{3}$, then 
\begin{align*}
C_{j,k} & \subseteq C_{j-3,k} + \{ 3a, 2b, 2c \},\\
D_{j,k} & = D_{j-1,k} + \{ a \}.
\end{align*}
\end{itemize}

Now we notice that 
\begin{align*}
D_{3k+1,k} & = \left[(6k+4)a + \frac{a}{2} + 1, (6k+5)a - 1  \right]
\end{align*}
and
\begin{align*}
[(6k+5)a, (6k+6)a] \subseteq C_{3k-1,k} + \{3a, 2b, 2c \}.
\end{align*}
Hence $\Ap(G,a) \subseteq [a,(6k+6)a]$ and $E \subseteq G$.

\item All inequalities follow immediately from the definition of the sets.

\item We notice that $H_{16,k}$ is a co-finite subset of $\N$.

Now we take $\{x, y \} \subseteq H_{16,k}$ and show that $x+y \in H_{16,k}$ dealing with different cases.

\begin{itemize}
\item \emph{Case 1:} $x \in E$ or $y \in E$. Then $x+y \in E$.
\item \emph{Case 2:} $\{x, y \} \subseteq [(3k+3)a, \infty[$. Then $x+y \in E$.
\item \emph{Case 3:} $x < (3k+3)a$ and $y < (6k+4)a  + \frac{a}{2} + 1$. 

We have that
\begin{equation*}
x \in A_{i_1,k} \cup B_{i_1,k}
\end{equation*}
for some $i_1 \in [0,3k+2]$. We analyse in detail just the case $x \in A_{i_1,k}$.

If $x + y \in E$, then we are done.

Now we suppose that $x + y \not \in E$ and consider some subcases.

\begin{itemize}
\item If $y \in A_{i_2,k}$ for some $i_2 \in [0,3k+2]$, then 
\begin{equation*}
x + y \in A_{i_1+i_2,k}
\end{equation*}
if $i_1 + i_2 \leq 3 k +2$, else 
\begin{equation*}
x+y \in C_{i_1+i_2-(3k+3),k} \cup D_{i_1+i_2-(3k+3),k}.
\end{equation*}

\item If $y \in B_{i_2,k}$ for some $i_2 \in [0,3k+2]$, then 
\begin{equation*}
x + y \in B_{i_1+i_2,k}
\end{equation*}
if $i_1 + i_2 \leq 3 k +2$, else 
\begin{equation*}
x+y \in D_{i_1+i_2-(3k+3),k} \cup C_{i_1+i_2+1-(3k+3),k}.
\end{equation*}

\item If $y \in C_{j,k}$ for some $j \in [0,3k+1]$, then 
\begin{equation*}
x + y \in C_{i+j,k} \cup D_{i+j,k}.
\end{equation*}

\item If $y \in D_{j,k}$ for some $j \in [0,3k+1]$, then 
\begin{equation*}
x + y \in D_{i+j,k} \cup C_{i+j+1,k}.
\end{equation*}
\end{itemize}

\end{itemize} 

\item For any $i \in [1,3k+2]$ we have that 
\begin{align*}
|A_{i,k}| & = 3 \left\lfloor \frac{i}{3} \right\rfloor + 1 \equiv 1 \Mod{3},\\
|B_{i,k}| & = 3 \left\lfloor \frac{i-1}{3} \right\rfloor + 2 \equiv 2 \Mod{3}.
\end{align*}
Therefore $|A_{i,k} \cup B_{i,k}| \equiv 0 \Mod{3}$.

The sequence formed by the greatest element of $A_{i,k}$ and the two smallest elements of $B_{i,k}$ reads as follows (modulo $3$):
\begin{equation*}
\left( \overline{ia+ 6 \left\lfloor \frac{i}{3} \right\rfloor}, \overline{(i-1)a+ b}, \overline{(i-1) a + b + 2} \right) = (\overline{2i}, \overline{2i+2}, \overline{2i+1}).
\end{equation*}
All the other elements of $\overline{g \cap A_{i,k}}$ and $\overline{g \cap B_{i,k}}$ can be obtained via concatenations of $3$-permutations. 

As regards the sets $C_{j,k}$ and $D_{j,k}$, we have that 
\begin{align*}
|C_{j,k}| & \equiv 0 \Mod{3},\\
|D_{j,k}| & \equiv 0 \Mod{3}.
\end{align*}
Moreover, the elements of $\overline{g \cap C_{j,k}}$ and $\overline{g \cap D_{j,k}}$ are given by concatenations of $3$-permutations.

Therefore $H_{16,k}$ is a $3$-permutation semigroup.
\end{enumerate}
\end{proof}

\section{Open questions}\label{sec_n_pns}
\subsection{Towards the classification of $3$-permutation semigroups}
In Section \ref{sec_3_pns} we have constructed 16 families of $3$-permutation semigroups.

We notice that if $G$ is a $3$-permutation semigroup and $g_1 < g_2 < g_3$ are the three smallest positive elements of $G$, then 
\begin{align*}
g_2 \leq 2 g_1,\\
g_3 \leq 3 g_1.
\end{align*} 

We have written a  GAP function \verb+persgp(k,m,n)+, which returns the $k$ generators $g_1 < \dots < g_k$ of a $k$-permutation semigroup such that $m \leq g_1 < \dots < g_k \leq n$. 

In accordance with the remark above, we can find all  $3$-permutation semigroups of a fixed multiplicity $M > 0$ running the function \verb+persgp(k,m,n)+ with parameters $k =3$, $m \leq M$ and $n \geq 3M$.

For example, we can get the list of all $3$-permutation semigroups having multiplicity not greater than $11$ running \verb+persgp(3,1,33)+. The generating set of any such semigroup is one of the sets in Table 1.

In Table 2 we list all generating sets of $3$-permutation semigroups having multiplicity between $12$ and $35$. For any such a semigroup we also specify the family it belongs to. Such a list has been obtained running  \verb+persgp(3,12,105)+.

\newpage

\begin{minipage}[t]{0.3\textwidth}
\begin{displaymath}
\centering
\begin{array}{|c|c|c|}
\hline 
{\textbf{Table 1}}
\\
\hline
{S}
\\
\hline
\{ 1, 2 , 3 \} \\ 
\hline 
\{ 2, 3, 4 \} \\ 
\hline
\{3, 4, 5 \} \\ 
\hline 
\{4, 5, 6 \} \\ 
\hline 
\{5, 6, 7 \} \\ 
\hline 
\{5, 7, 9 \} \\ 
\hline 
\{6, 7, 11 \} \\ 
\hline 
\{7, 8, 12 \} \\ 
\hline 
\{7, 9, 11 \} \\ 
\hline 
\{7, 11, 12 \} \\ 
\hline 
\{8, 9, 10 \} \\ 
\hline 
\{8, 10, 15 \} \\ 
\hline 
\{8, 13, 15 \} \\ 
\hline 
\{9, 10, 17 \} \\ 
\hline 
\{9, 11, 16 \} \\ 
\hline 
\{10, 11, 15 \} \\ 
\hline 
\{10, 17, 18 \} \\ 
\hline 
\{11, 12, 13 \} \\ 
\hline 
\{11, 13, 21 \} \\ 
\hline 
\{11, 15, 16 \} \\ 
\hline 
\{11, 16, 18 \} \\ 
\hline 
\{11, 19, 21 \} \\
\hline 
\end{array} 
\end{displaymath} 
\end{minipage}
\begin{minipage}[t]{0.65\textwidth}
\begin{displaymath}
\begin{array}{|c||c||c|}
\hline
\multicolumn{3}{|c|}{\textbf{Table 2}}\\
\hline
{S} & {S} & \text{Family}\\
\hline 
\{12, 13, 23  \} & \{24, 25, 47 \} & H_{1,k} \\ 
\hline 
\{13, 14, 21 \} & \{ 25, 26, 39 \} & H_{2,k}\\ 
\hline 
\{ 13, 15, 17 \} & \{25, 27, 29\} & H_{4,k}\\ 
\hline 
\{13, 20, 21\} & \{25, 38, 39\} & H_{3,k}\\ 
\hline 
\{13, 20, 24\} & \{25, 44, 48\} & H_{5,k}\\ 
\hline 
\{13, 23, 24\} & \{25, 47, 48\} & H_{6,k}\\ 
\hline 
\{14, 15, 16\} & \{26, 27, 28\} & H_{7,k}\\ 
\hline 
\{14, 16, 21\} & \{26, 28, 39\} & H_{8,k}\\ 
\hline 
\{14, 25, 27\} & \{26, 49, 51\} & H_{9,k}\\ 
\hline 
\{15, 16, 29\} & \{27, 28, 53\} & H_{1,k}\\ 
\hline 
\{15, 17, 28\} & \{27, 29, 52\} & H_{10,k}\\ 
\hline 
\{16, 17, 24\} & \{28, 29, 42\} & H_{11,k}\\ 
\hline 
\{16, 21, 23\} & \{28, 39, 41\} & H_{12,k}\\ 
\hline 
\{16, 29, 30\} & \{28, 53, 54\} & H_{6,k}\\ 
\hline 
\{17, 18, 19\} & \{29, 30, 31\} & H_{7,k}\\ 
\hline 
\{17, 24, 25\} & \{29, 42, 43\} & H_{13,k}\\ 
\hline 
\{17, 28, 30\} & \{29, 52, 54\} & H_{14,k}\\ 
\hline 
\{17, 31, 33\} & \{29, 55, 57\} & H_{9,k}\\ 
\hline 
\{18, 19, 35\} & \{30, 31, 59\} & H_{1,k}\\ 
\hline 
\{19, 20, 30\} & \{31, 32, 48\} & H_{2,k}\\ 
\hline 
\{19, 21, 23\} & \{31, 33, 35\} & H_{4,k}\\ 
\hline 
\{19, 29, 30\} & \{31, 47, 48\} & H_{3,k}\\ 
\hline 
\{19, 32, 36\} & \{31, 56, 60\} & H_{5,k}\\ 
\hline 
\{19, 35, 36\} & \{31, 59, 60\} & H_{6,k}\\ 
\hline 
\{20, 21, 22\} & \{32, 33, 34\} & H_{7,k}\\ 
\hline 
\{20, 22, 33\} & \{32, 34, 51\} & H_{15,k}\\ 
\hline 
\{20, 31, 33\} & \{32, 49, 51\} & H_{16,k}\\ 
\hline 
\{20, 37, 39\} & \{32, 61, 63\} & H_{9,k}\\ 
\hline 
\{21, 22, 41\} & \{33, 34, 65\} & H_{1,k}\\ 
\hline 
\{21, 23, 40\} & \{33, 35, 64\} & H_{10,k}\\ 
\hline 
\{22, 23, 33\} & \{34, 35, 51\} & H_{11,k}\\ 
\hline 
\{22, 41, 42\} & \{34, 65, 66\} & H_{6,k}\\ 
\hline 
\{23, 24, 25\} & \{35, 36, 37\} & H_{7,k}\\ 
\hline 
\{23, 33, 34\} & \{35, 51, 52\} & H_{13,k} \\ 
\hline 
\{23, 40, 42\} & \{35, 64, 66\} & H_{14,k}\\ 
\hline 
\{23, 43, 45\} & \{35, 67, 69\} & H_{9,k}\\ 
\hline 
\end{array}
\end{displaymath} 
\end{minipage}

\vspace{0.5cm}

The data above and other experimental data for $3$-numerical semigroups seem to support the following conjecture.

\begin{conjecture}
Any $3$-numerical semigroup having multiplicity at least $12$ belongs to one of the family $H_{1,k} - H_{16,k}$ of Section \ref{sec_3_pns}.
\end{conjecture}

\subsection{On $n$-permutation semigroups}
In Lemma \ref{family_n_1} we construct a family of $n$-permutation semigroups for any $n \geq 3$. We notice in passing that the $3$-permutation semigroups of the family $H_{6,k}$ belong to this more general family of $n$-permutation semigroups. 

The experimental evidence suggests that many more $n$-numerical semigroups exist, but we think that a classification of all of them is (at the moment) out of reach.
 
\begin{lemma}\label{family_n_1}
Let $n \geq 3$ be an integer, $k$ a positive integer and $a:=nk+1$.
Let $S := \{ a \} \cup [2a - n, 2 a -2]$ and $G := \langle S \rangle$.

Let $ H:= \{ 0 \} \cup \left( \cup_{i=0}^{k-1} (A_{i,k}  \cup B_{i,k}) \right) \cup [2ka, \infty[$, where
\begin{align*}
A_{i,k} & := \{(2i+1) a \} \cup [(2i+2) a - (i+1) n, (2i +2) a - 2],\\
B_{i,k} & := \{(2i+2) a \} \cup [(2i+3) a - (i+1) n, (2i +3) a - 2],
\end{align*}
for any $i \in [0,k-1]$.

Then the following hold.
\begin{enumerate}
\item $H$ is a numerical semigroup.
\item $G = H$.
\item $H$ is a $n$-permutation semigroup.
\end{enumerate}
\end{lemma}

\begin{proof}
\begin{enumerate}[leftmargin=*]
\item Let $\{i_1, i_2 \} \subseteq \N$ and $i_3 := i_1 + i_2 + 1$.
If
\begin{align*}
x & \in A_{i_1,k} \cup B_{i_2,k},\\
y & \in A_{i_{2},k} \cup B_{i_2,k},
\end{align*}
then $x+y$ belongs to one of the sets in rows $2-3$, columns $2-3$ of the following table.  

\begin{displaymath}
\begin{array}{|c|c|c|}
\hline
& A_{i_2,k} & B_{i_2,k}\\
\hline
A_{i_1,k} & A_{i_3,k} \cup B_{i_3-1,k} & A_{i_3,k} \cup B_{i_3,k}\\
\hline
B_{i_1,k} & & A_{i_3+1,k} \cup B_{i_3,k} \\
\hline
\end{array}
\end{displaymath}

Since $H$ is co-finite, we conclude that $H$ is a numerical semigroup.

\item First we notice that $S \subseteq H$.

Now we prove that $H \subseteq G$, discussing separately some cases.

\begin{itemize}
\item \emph{Case 1:} $A_{i,k} \subseteq G$ for any $i \in [0,k-1]$.

For any $j \in \N$ we have that $j a \in G$.

Moreover, for any $i \in [0,k-1]$ we have that 
\begin{equation*}
[(2i+2)a - (i+1) n, (2i+2) a - (2i+2)] \subseteq \langle [2a-n, 2a-2] \rangle
\end{equation*}
according to Lemma \ref{lem_arith_seq}.

We prove by induction on $i \in [0,k-1]$ that
\begin{equation*}
(2i+2) a - j \in G
\end{equation*}
for any $j \in [2,2i+1]$.

If $i = 0$, then there is nothing to prove.

If $i > 0$, then
\begin{align*}
(2i+2) a - j & = (2(i-1)+2) a - j + 2 a. 
\end{align*}

Since $j \leq 2i + 1 \leq 3i \leq n i$, we have that
\begin{align*}
(2(i-1)+2) a - j \geq (2(i-1)+2) a - i n.
\end{align*}
Therefore $(2(i-1)+2)a - j \in A_{i-1,k}$. Since $A_{i-1,k} \subseteq G$ by inductive hypothesis, we conclude that  $(2i+2) a - j \in G$.

\item \emph{Case 2:} $B_{i,k} \subseteq G$ for any $i \in [0,k-1]$.

We notice that $B_{i,k} = A_{i,k} + \{ a \}$ for any $i$. Hence $B_{i,k} \subseteq G$ for any $i$.

\item \emph{Case 3:} $[2ka, \infty[ \subseteq G$.

First we notice that
\begin{align*}
A_{k-1,k} & = [(2k-1) a, 2 k a - 2],\\
B_{k-1,k} & = [2ka, (2k+1) a-2].
\end{align*}

Then we observe that
\begin{align*}
(2k+1) a - 1 & = (2k-1) a + 1 + (2a-2) \subseteq A_{k-1,k} + A_{0,k}.
\end{align*}

Therefore $[2ka, (2k+1) a - 1] \subseteq G$, namely $\Ap(G,a) \subseteq [a,(2k+1)a-1]$. 

Hence $[2ka, \infty[ \subseteq G$.
\end{itemize}

\item We notice that 
\begin{equation*}
|A_{i,k}| = |B_{i,k}| = (i+1) n
\end{equation*}
for any $i \in [0,k-1]$, namely $n$ divides the cardinality of any set $A_{i,k}$ and $B_{i,k}$.

Since the elements of the sets $A_{i,k}$ and $B_{i,k}$ are obtained through concatenations of $n$-permutations, we conclude that $H$ is a $n$-permutation semigroup.

\end{enumerate}
\end{proof}

\bibliography{Refs}
\end{document}